\documentclass[10pt]{article}
\usepackage{amsmath,amssymb,amsfonts,amsthm,mathtools,xfrac,extarrows,enumitem,bbold}
\usepackage{epstopdf}
\usepackage{graphicx, float}
\usepackage[T1]{fontenc}
\usepackage{amscd}
\usepackage[active]{srcltx}
\usepackage[font={small}]{caption}
\usepackage{epstopdf}
\usepackage[T1]{fontenc}
\usepackage{amscd}
\usepackage{hyperref}
\newcommand{\N}{\mathbb{N}}
\newcommand{\R}{\mathbb{R}}

\def\nlongrightarrow{\;\;\;/\mkern-25mu\longrightarrow}

\def\nLongleftarrow{\quad\;/\mkern-25mu\Longleftarrow \quad}
\def\nLongrightarrow{\quad\;/\mkern-25mu\Longrightarrow \quad}

\newcommand\extrafootertext[1]{%
    \bgroup
    \renewcommand\thefootnote{\fnsymbol{footnote}}%
    \renewcommand\thempfootnote{\fnsymbol{mpfootnote}}%
    \footnotetext[0]{#1}%
    \egroup
}

\newtheorem{thm}{Theorem}
\newtheorem{lemma}{Lemma}[section]

\newtheorem{remark}[lemma]{Remark}
\newtheorem{proposition}[lemma]{Proposition}
\newtheorem{definition}[lemma]{Definition}

\title{\vspace{-1.0cm}Prescribing scalar curvatures: on the negative Yamabe case}

\author
{
Martin Mayer
\\
\small{Scuola Superiore Meridionale,
Largo San Marcellino 10, 80138 Napoli, Italia
}
\\
\&
\\
Chaona Zhu
\\
\small{
Dipartimento di Matematica dell'Universit\`a degli studi di Roma "Tor Vergata"
}
\\
\small{
Via della Ricerca Scientifica 1, Roma, Italia
}
}

\normalsize

\begin{document}
\maketitle
\begin{abstract}
The problem of prescribing conformally the scalar curvature on a closed Riemannian manifold
of negative Yamabe invariant is always solvable, when the function $K$ to be prescribed is strictly
negative, while sufficient and necessary conditions are known for $K\leq 0$.
For sign changing $K$ Rauzy \cite{Rauzy} showed solvability,
provided $K$ is not too positive.
We revisit this problem in a different variational context,
thereby recovering and quantifying the principal existence result of Rauzy
and show under additional assumptions,
that for a sign changing $K$ solutions to the conformally prescribed scalar curvature problem,
while existing, are not unique.
\end{abstract}

{\footnotesize
\begin{center}
nonlinear analysis
\\
\end{center}
}

\tableofcontents

\section{Introduction}\label{Sec_Introduction}
Let $M=(M^n, g_0)$ be a closed Riemannian manifold with $n\geq3$.
We consider the classical conformally prescribed scalar curvature problem,
i.e., given a smooth function $K$ on $M$, we ask
for the existence of a conformal metric $g$ to $g_0$,
whose  scalar curvature is $K$.

If we denote by $g=g_{u}=u^{\frac{4}{n-2}}g_0$ with $u>0$ a conformal metric to $g_{0}$,
this problem is equivalent to finding a positive solution $u>0$ of the equation
\begin{equation}\label{Y}
L_{g_0}u
=
-c_n\Delta_{g_0} u+R_{g_0}u
=
Ku^{\frac{n+2}{n-2}}, \; c_n=\frac{4(n-1)}{n-2}.
\end{equation}
Here $R_{g_0}$ denotes the scalar curvature with respect to the metric $g_0$.
In particular, when $K$ is constant,
\eqref{Y} reduces to the Yamabe problem,
which has been, as is well known, completely solved by the works of Yamabe, Trudinger, Aubin and Schoen.

The prescribed scalar curvature problem for non constant $K$ has been widely studied as well,
especially in case of a positive Yamabe invariant, in particular on the standard sphere $S^{n}$,
see for instance \cite{MM4} and the references therein.
Here we are interested in the case of a negative Yamabe invariant, i.e. when
$$
Y(M)
=
\inf_{\substack{u\in H^{1}(M)\\ u>0}}
\frac
{\int_ML_{g_0}uud\mu_{g_0}}
{(\int_Mu^{\frac{2n}{n-2}}d\mu_{g_0})^{\frac{n-2}{n}}}
<0,
$$
and we refer to \cite{Aubin_Book} for a comprehensive introduction.
By the resolution of the Yamabe problem \cite{Kazdan_Warner_JDE}
we may assume $R_{g_0}\equiv-1$, whence
$$L_{g_{0}}=-c_{n}\Delta_{g_{0}}-1$$
and the constant functions become the first and a negative eigenspace of $L_{g_{0}}$.

\

If $K<0$, then \eqref{Y} is always solvable, and still,
if $K\leq 0$, solutions are unique \cite{Kazdan_Warner_JDE}.
Moreover
a necessary condition for solvability of \eqref{Y} is
\begin{equation}\label{necessary_condition_nu_1_positive}
\nu_{1}(\Omega_{K})=\nu_{1}(L_{g_{0}},\Omega_{K})>0,
\end{equation}
as was first proved by Rauzy \cite{Rauzy},
where
$\Omega_K=\{K \geq 0\}$,
\begin{equation*}%\label{eig}
\nu_{1}(L_{g_{0}},\Omega_K)
=
\sup_{\Omega_{K} \subset \Omega \; \text{smooth } \; }
\nu_{1}(L_{g_{0}},\Omega)
\end{equation*}
is the first Dirichlet eigenvalue
and here for a smooth subset $\Omega \subset M$
\begin{equation*}%\label{eig_explained}
\nu_{1}(L_{g_{0}},\Omega)
=
\inf_\mathcal{A}\frac{\int_{\Omega}L_{g_0}uud\mu_{g_0}}{\int_{\Omega}u^2d\mu_{g_0}}
,
\;
\mathcal{A}=\{u\in C_{0}^\infty(\Omega) \; : \; u>0 \; \text{ in } \; \Omega\ \}.
\end{equation*}
And actually,
if $K\leq 0$, the necessary condition $\nu_{1}(\Omega_{K})>0$ is also
sufficient to guarantee solvability, see \cite{Ouyang,Rauzy,Tang,Vazquez_Veron}.
Furthermore an additional necessary condition,
which is automatically satisfied for $0\not \equiv K \leq 0$,
is the positivity of the unique solution $\bar{w}>0$ of
\begin{equation}\label{linear_control_equation}
\mathcal{L}_{g_{0}}\bar{w}=-(n-1)\Delta_{g_{0}}\bar{w}+\bar{w}=-K,
\end{equation}
see \cite{Kazdan_Warner_JDE}, and in particular necessarily
$\int K d\mu_{g_{0}}<0$.
Indeed
$\bar{u}=u^{-\frac{4}{n-2}}$ for a solution $u>0$ of \eqref{Y}
is a subsolution to \eqref{linear_control_equation} and thus the maximum principle tells us,
that necessarily $\bar{u}\leq \bar{w}$.
Finally, when $K$ changes sign, Rauzy \cite{Rauzy} used a subcritical approximation argument
to obtain solvability under a smallness assumption on $\sup K$,
which later on  Aubin\&Bismuth \cite{Aubin_Bismuth} quantified, see
Remark \ref{rem_Aubin_Bismuth_Error} below and for $n=2$ the analogous work \cite{Bismuth_Dimension_Two} by Bismuth.

\

In view of these results we will study for sign changing $K$
the conformally prescribed scalar curvature problem in a variational setting \textit{different} from the
one used by Rauzy \cite{Rauzy} and assume throughout the necessary conditions
\begin{enumerate}[label=(\roman*)]
\item $\nu_{1}(\Omega_{K})>0$
\item $\int_{M}Kd\mu_{g_{0}}<0$.	
\end{enumerate}
Let us first introduce some notations. Let
$$X=\{u>0\}\cap\{r<0\}\cap\{k<0\}\cap\{\Vert u\Vert _{L^{\frac{2n}{n-2}}}=1\}
\subset C^{\infty}(M)
,$$
where
\begin{equation}\label{r_and_k_definitions}
r=r_{g_{u}}=\int_MR_{g}d\mu_g=\int_ML_{g_0}uud\mu_{g_0} \;
\text{and} \; k=k_{g_{u}}=\int_MKu^{\frac{2n}{n-2}}d\mu_{g_0}.
\end{equation}
Clearly $X\neq \emptyset$, as some constant function is in $X$,
and we consider
\begin{equation*}%\label{J}
J=\frac{-k}{(-r)^{\frac{n}{n-2}}}>0
\end{equation*}
as a scaling invariant functional on $X$ with derivative
\begin{equation*}%\label{partial_J}
\partial J(u)
=
\frac{2^*}{(-r)^{\frac{n}{n-2}}}\bigg(\frac{-k}{-r}\, L_{g_0}u-Ku^{\frac{n+2}{n-2}}\bigg)
,\;
2^*=\frac{2n}{n-2}
\end{equation*}
and a Yamabe type flow
\begin{equation}\label{flow_for_J}
\partial_{t}u
=
-(\frac{-k}{-r}R-K)u
=
-u^{-\frac{4}{n-2}}(\frac{-k}{-r}L_{g_{0}}u-Ku^{\frac{n+2}{n-2}}).
\end{equation}
In this way $J$ becomes a variational functional
and $X$ a variational space,
as any solution to $\partial J=0$ on $X$ is a solution to \eqref{Y}.
Note, that the choice of $X$ is somewhat natural,
since for $K\leq 0$ any normalized solution
to the conformally prescribed scalar curvature problem must be in $X$.
On the other hand $J$ is not necessarily bounded from below
and, flowing along \eqref{flow_for_J}, while
the factors $-k,-r>0$ are readily uniformly bounded away from infinity on $X$,
generally
\begin{equation}\label{r_k_or_k_r_not_bounded}
0<-k\longrightarrow 0
\; \text{ or } \;
0<-r \longrightarrow 0
\end{equation}
may happen and it is natural to ask,
how to prevent this scenario.

\

We will show first, that under conditions similar to,
but generally  weaker than those of Aubin\&Bismuth \cite{Aubin_Bismuth}
a certain integral condition holds,
to which we refer as a global A-B-inequality, namely
$$
\Vert u \Vert_{H^{1}}^{2}
\leq
Ar
+
B\vert k \vert^{\frac{n-2}{n}}
\; \text{ for all } \;
u\in H^{1}(M),
$$
see Proposition \ref{prop_A_B_inequality_from_A_B_conditions},
and second, that an A-B-inequality holding on $X$ guarantees $\inf_{X} J>0$ and
\eqref{r_k_or_k_r_not_bounded} does not occur,
whence $J$ becomes an energy and the flow generated by \eqref{flow_for_J} can be used for variational arguments.
In particular, by choosing appropriate initial data,
we find a solution of \eqref{Y},
which is a global minimizer of the functional $J$ on $X$.
Note, that a minimizing property or saddle point structure of
the solution obtained via approximation and variational means by Rauzy or via perturbation arguments by
Aubin\&Bismuth
is unknown, while here we argue on the critical equation directly
and do not rely on the method of sub- and supersolutions for instance and in
contrast to e.g. \cite{Aubin_Bismuth}, \cite{Kazdan_Warner_JDE}.

\begin{thm}\label{thm_minimize_J}
If an A-B-inequality holds
on some sublevel
$$J^{L}=\{ J\leq L \} \neq \emptyset,$$
then $J$ admits a global minimizer on $X$,
which is a solution of equation \eqref{Y}.
\end{thm}
For the proof see Section \ref{section_solvability_by_minimization}.
Combined with
Proposition \ref{prop_A_B_inequality_from_A_B_conditions}, Theorem \ref{thm_minimize_J} quantifies
the smallness assumption in  Rauzy \cite{Rauzy}.
\begin{remark}\label{rem_Aubin_Bismuth_Error}
Let us review the Aubin\&Bismuth result in the corresponding notations
of Remark 6.13 in \cite{Aubin_Book} and Theorem 6 in \cite{Aubin_Bismuth},
namely solvability of
$$
L_{g_{0}}u=fu^{\frac{n+2}{n-2}},\; u>0
\; \text{ for } \;
f \in \mathcal{F}_{\alpha,K}
=
\{ f\in C^{\alpha}(M) \; : \;
\{ f\geq 0 \} = K
 \},
$$
where $0<\alpha<1$ and $K \subset M$ are fixed.
\begin{enumerate}[label=(\roman*)]
\item[a)] In Remark 6.13 in \cite{Aubin_Book} solvability is claimed, provided, that
for some smooth
\begin{equation} \label{Omega_in_Aubin_Bismuth}
\Omega \supset K \; \text{ with  } \; \lambda(\Omega) > -\tilde{R}
\end{equation}
there holds
$$\sup f \leq C(K) \inf_{M\setminus \Omega}(-f),$$
where $C(K)$ supposedly depends on
$K=\{ f\geq 0 \}$ only.
But this is not substantiated by the reference to \cite{Aubin_Bismuth}, see Theorem 6 there,
whose statement requires specific choices of neighbourhoods
$\Omega \supset \theta \supset K$, see b) below, and thus
$\Omega$ as in \eqref{Omega_in_Aubin_Bismuth} is not arbitrary.

\item[b)]
In Theorem 6 in \cite{Aubin_Bismuth} solvability is claimed, provided
\begin{equation}\label{Smallness_Assumption_In_Aubin_Bismuth}
\sup K \leq C(K)\inf_{M\setminus \theta} (-f),
\end{equation}
where according to the Definition before Theorem 6
\begin{equation}\label{theta_and_Omega_in_Aubin_Bismuth}
\Omega \supset \supset \theta \supset \supset K
\end{equation}
are smooth neighbourhoods of $K$, which in particular satisfy
$$
\frac{1}{2}\lambda(K) < \lambda(\theta) < \lambda(K)
\; \text{ and } \;
\frac{1}{2}\lambda(\theta) < \lambda(\Omega) < \lambda(\theta)
$$
for the first Dirichlet eigenvalues $\lambda$ of these sets, whence specific choices for $\theta$ and $\Omega$
are required in contrast to Proposition \ref{prop_A_B_inequality_from_A_B_conditions} below.

\item[c)]
Moreover $C(K)$ in \eqref{Smallness_Assumption_In_Aubin_Bismuth} is supposed to depend on
$K=\{ f\geq 0 \}$ only, which is surprising, since the smallness constant
in our Proposition \ref{prop_A_B_inequality_from_A_B_conditions} depends on a distance
corresponding to  $d(\partial\theta,\partial\Omega)$, cf. \eqref{theta_and_Omega_in_Aubin_Bismuth}.
And in fact the upper bound of $\psi$ in Proposition 1 of \cite{Aubin_Bismuth},
claimed to be
$$\sup \psi \leq C(K),$$
tends to infinity, when $d(\partial \Omega,\partial\theta)\longrightarrow 0$.
To be precise, if $d(\partial \Omega,\partial\theta)\longrightarrow 0$, then $\inf\, \varphi \longrightarrow 0$
from  (15) and the definition of $\varphi$ two  lines below, while
$$u^{+}=\xi \varphi \; \text{ for a constant } \; \xi>0
\; \text{ and } \;
u^{-}=\vert R \vert^{\frac{n-2}{4}}
$$
shall act as a super- and subsolution respectively.
To apply the method of sub- and supersolutions we then require
$$u^{-}\leq u^{+} \; \text{ pointwise and thus } \; \xi\longrightarrow \infty,$$
as $sup \,\varphi \nlongrightarrow 0$. So a uniform upper bound for $\psi$ is not feasible.
 \end{enumerate}
\end{remark}
\noindent
Either way,
while Aubin\&Bismuth in \cite{Aubin_Bismuth} qualify via \eqref{Smallness_Assumption_In_Aubin_Bismuth}
the smallness condition of Rauzy in \cite{Rauzy},
Proposition \ref{prop_A_B_inequality_from_A_B_conditions}
quantifies the smallness constant $C(K)>0$ in \eqref{Smallness_Assumption_In_Aubin_Bismuth}.
Concerning non existence and
recalling \eqref{necessary_condition_nu_1_positive} and \eqref{linear_control_equation} we have
\begin{lemma}\label{lem_conditions_for_existence}
There holds
$$
\nu_{1}(\Omega_{K})>0
\begin{matrix}
\nLongleftarrow\\
\nLongrightarrow
\end{matrix}	
\bar{w}>0
\;\; \Longrightarrow \;
\int_{M} Kd\mu_{g_{0}}<0.
$$
\end{lemma}	
Thus (i) or (ii) above are alone  not sufficient to guarantee solvability.
To see Lemma \ref{lem_conditions_for_existence}, whose demonstration we believe to be instructive,
first note, that the last implication
is immediate from testing \eqref{linear_control_equation} against a constant.
Secondly, to see $\bar{w} >0\nLongrightarrow \hspace{-4pt}\nu_{1}(\Omega_K)>0$, note, that $\bar{w}>0$ for $0 \not \equiv K\leq 0$ by the maximum principle,
while we may choose $\Omega_{K}=M\setminus B_{\epsilon}(x_{0})$.
Then for suitable
$0\leq\eta_{\epsilon,x_{0}}\leq 1$
with
$\eta_{\epsilon,x_{0}} \equiv 0$ on $B_{\epsilon}(x_{0})$ we have
$$\int_{M}L_{g_{0}}\eta_{\epsilon,x_{0}}\eta_{\epsilon,x_{0}}<0,$$
whence $\nu_{1}(\Omega_{K})<0$.
Finally we may even construct $K\in C^{\infty}(M)$, for which (i) and (ii) hold,
but $\bar{w} \not >0$.
Indeed $\mathcal{L}_{g_{0}}>0$ has a Green's function $G_{\mathcal{L}_{g_{0}}}$ satisfying
$$
\inf_{(M\times M) \setminus \{  diag(M) \} }G_{\mathcal{L}_{g_{0}}}>0
\; \text{ with a principal term} \;
G_{\mathcal{L}_{g_{0}}}(x,y)\simeq d_{g_{0}}^{2-n}(x,y)
$$
and we consider for $\min\{2,n/2\}<p<n$, $0<\epsilon \ll 1$ and $ \epsilon \lambda^{n-p} \gg 1$
$$K=-\epsilon+\eta_{\epsilon,x_{0}}(\frac{\lambda}{1+\lambda^{2}d^{2}_{g_{0}}(x_{0},x)})^{p}$$
with a suitable cut-off function $\eta_{\epsilon,x_{0}}$,
$\eta_{\epsilon,x_{0}} \equiv 1$ on $B_{\epsilon}(x_{0})$.
We then easily check
$$\nu_{1}(\Omega_{K})>0 \; \text{ and } \; \int_{M}Kd\mu_{g_{0}}<0,$$
while
$$\bar{w}(x_{0})=-G_{\mathcal{L}_{g_{0}}}(x_{0},\cdot)*K\simeq -\lambda^{p-2}<0.$$
In particular the last argument shows,
that some kind of smallness assumption to guarantee solvability is natural.

\

On the other hand, once solvability of \eqref{Y} is given,
we naturally ask for uniqueness,
which, as we recall, holds true in case $K\leq 0$.
Surprisingly, as our second theorem shows,
we may \textit{either} lose existence \textit{or} uniqueness,
when passing from $K < 0$ strictly negative to $K$ sign changing.

In the latter case, when we lose uniqueness in that passage,
but not existence,
we find at least two solutions,
one inducing a totally \textit{negative} scalar curvature $r=k<0$,
see \eqref{r_and_k_definitions},
while more surprisingly the other solution induces a totally \textit{positive} one, see Theorem \ref{thm_second_solution} below,
although, fixing some suitable $0\not\equiv K_-\leq 0$,
we may choose $0\not \equiv K_{+}\geq 0$ such,
that for $K=K_{+} + K_{-}$ the positive maximum of $K$ is as small as we wish.

\

For the sake of simple statements we say, that $Cond_{n}$ holds at $a\in M$, if
$$
\quad \quad
(Cond_{n})
\quad
\begin{cases}
\, \,
3\leq n \leq 5 \; \text{ and no restrictions on $a\in M$ are required}
\\
\,\,\;\;\;\;\;\,\; n\geq 6 \; \text{ and $M$ is locally conformally flat around $a\in M$.}
\end{cases}
$$

In particular $Cond_{n}$ is satisfied for all $a\in M$, if $3\leq n\leq 5$.

\begin{thm}\label{thm_second_solution}
Suppose, that $L_{g_{0}}$ is invertible, and consider
a sign changing function $K\in C^{\infty}(M)$,
for which a global  A-B-inequality holds. Then, if
$Cond_{n}$ is satisfied at some
$a\in \{ K=\max K \} $ and
$$
\nabla^{l}K(a)=0 \; \text{ for all } \; \frac{n-2}{2} \geq l \in \N,
$$
 there exists $C=C(A,B)$ such, that
the conformally prescribed scalar curvature problem
admits at least two solutions
$u_{0},u_{1}$, provided
$\sup K \leq C$,
in which case
$$
r_{u_{0}},k_{u_{0}}<0
\; \text{ and } \;
r_{u_{1}},k_{u_{1}}>0.
$$
\end{thm}	
For the proof see Section \ref{sec_compactness_and_existence}.
Some remarks on Theorem \ref{thm_second_solution} are in order.
\begin{remark}
\begin{enumerate}[label=(\roman*)]
\item 	
The function $u_{0}$ is the minimizer from Theorem \ref{thm_minimize_J}, while
the second solution $u_{1}$ is a minimizer of
$
I=\frac{r}{k^{\frac{n-2}{n}}}
$
on the natural domain
$$
Y=
\{ r>0 \} \cap \{ k>0 \} \cap \{ u>0 \} \cap \{ \Vert \cdot \Vert_{L^{\frac{2n}{n-2}}}=1 \}
\subset C^{\infty}(M).
$$
Clearly $Y=\emptyset$ for $K\leq 0$ and, if an A-B-inequality holds, then $\inf_{Y}I>0$.
\item
Let $J_{\epsilon}=J_{K_{\epsilon}}$, where
$K_{\epsilon}=K_{0}+\epsilon K_{1}$, $K_{0} \leq 0$ and $\sup K_{1}>0$.
Suppose, that an A-B-inequality holds for $K_{\epsilon_{0}}$,
whence for all $0\leq \epsilon \leq \epsilon_{0}$ in particular
the same A-B-inequality holds for $K_{\epsilon}$, and thus
we can minimize $J_{\epsilon}$ from Theorem \ref{thm_minimize_J}.
Then, at least if Theorem \ref{thm_second_solution} is applicable,
we have
\begin{enumerate}[label=(\roman*)]
\item[1)] a unique solution $u_{0}$ for $\epsilon=0$, namely the minimizer of $J_{0}$
\item[2)] at least two solutions
$$u_{0,\epsilon}\neq u_{1,\epsilon},$$
namely the minimizers of $J_{\epsilon}$ and $I_{\epsilon}$ respectively.  	
\end{enumerate}
Thus $u_{1,\epsilon}$ must cease to exist in the passage $0<\epsilon \longrightarrow 0$. Indeed
$$
0<k_{u_{1,\epsilon}}
=
\int K_{\epsilon}u_{1,\epsilon}^{\frac{2n}{n-2}}d\mu_{g_{0}}
\leq
\epsilon\int K_{1}u_{1,\epsilon}^{\frac{2n}{n-2}}d\mu_{g_{0}}
\leq
\epsilon \max K_{1} \xrightarrow{\epsilon \to 0} 0,
$$
while from the validity of an A-B-inequality $r_{u_{1,\epsilon}}\neq o_{\epsilon}(1)$. Hence in fact
$$
\inf I_{\epsilon}=\frac{r_{u_{1,\epsilon}}}{k_{u_{1,\epsilon}}^{\frac{n-2}{n}}}
\xrightarrow{\epsilon \to 0} \infty.
$$
\item The dimensional dependence of Theorem \ref{thm_second_solution} is reminiscent of
the distinction in Theorem 2.1 and Theorem 2.3 in \cite{Escobar_Schoen}.
The differences are explained in terms of, for which dimensions vanishing of $\nabla^{l}K$is assumed,
and, if vanishing is assumed, to which degree $l$, played against the different scales of deconcentration,
i.e. the principal quantities, that prevent blow-up.
In case of Escobar\&Schoen this is the positive mass $H$,
when considering a single bubbling blow-up of type $u=\varphi_{a,\lambda}$,
for us it is the solution $u_{0}$ for a mixed type blow up $u=u_0 + \varphi_{a,\lambda}$.
The situation is as follows
\begin{center}
\begin{tabular}[t]{l|cc}
 & Escobar\&Schoen & Mayer\&Zhu \\
\hline
 no vanishing assumption  & $n=3$ & $3\leq n \leq 5$\\
 vanishing assumption & $n\geq 4$ & $n\geq 6$ \\
 vanishing degree & $l\leq n-2$ & $l\leq \sfrac{(n-2)}{2}$\\
 deconcentration term & $\sfrac{H(a)}{\lambda^{n-2}}$ & $\sfrac{u_{0}(a)}{\lambda^{\frac{n-2}{2}}}$
\end{tabular}
\end{center}
where $a\in \{K=\max K\}$ and in particular $\nabla K(a)=0$. On the other hand a derivative
$\nabla^l K(a)$ contributes in the corresponding energy expansion at most a quantity of order
$O(\sfrac{1}{\lambda^l} + \sfrac{1}{\lambda^n})$. Then, as is easy to see from the table above,
the deconcentration terms are in fact dominant.
We refer to \cite{Ahmedou_Ben_Ayed_Non_Simple_Blow_Up,Bahri_High_Dimensions,Mayer_Scalar_Curvature_Flow}
for related studies on the mixed type blow-up case.
\item
For Theorem \ref{thm_second_solution} to be meaningful,
invertibility of $L_{g_{0}}$ and local flatness
at least {\normalfont somewhere} on $M$ must be compatible.
Indeed, if $L_{g_{0}}$ is not invertible, let us consider
the finite dimensional space $ker( L_{g_{0}} )$, for each
eigenfunction $e_{i} \in ker L_{g_{0}}$ its nodal set
$N_{i}=\{ e_{i} = 0\}$ and
an open set $O_{i}\subset M$ with $N_{i}\cap O_{i} \neq \emptyset$ such,
that $O_{i}\cap O_{j}=\emptyset$ for $i\neq j$.
Following and slightly modifying the arguments in \cite{GHJL},
we then can perturb $g_{0}$ to gain invertibility,
while the perturbation leaves $g_{0}$
on $M \setminus \cup_{i}O_{i}$ unchanged.
The required localization of the perturbation in
\cite{GHJL} is based on adding
for each $i=1,\ldots,dim(ker(L_{g_{0}}))$ a suitable cut-off function $\eta=\eta_{i}$,
living on $O_{i}$,
to the definition of $h=h_{i}$ in line 5, page 799 in \cite{GHJL},
where $\psi=e_{i}$, i.e.
$$
h
=
\eta
\left(
c_{n}\psi
(
2\psi \mathring{\nabla}^{2}\psi
-
\psi^{2}\mathring{Ric}
)
+
(2c_{n}-1)(d\psi \otimes d\psi)^{o}
\right).
$$
We leave it to the reader to verify, that the arguments in \cite{GHJL}
proceed with only minor modifications.
In this way, if $M=(M^{n},g_{0})$ is of negative Yamabe invariant
and locally conformally flat on some $A\subset M\setminus \cup_{i}O_{i}$,
then we may slightly change the metric $g_{0}$ to gain invertibility
of $L_{g_{0}}$, while local conformal flatness on $A$ is unchanged and,
as the modification is only slight, the Yamabe invariant remains negative.

\end{enumerate}
\end{remark}	
Concerning Theorem \ref{thm_second_solution}
we also mention Rauzy \cite{Rauzy_Multiplicity} for complementary and
under much stronger assumptions \cite{Pistoia_Roman} for similar results.
Note, that extensions of Theorem \ref{thm_second_solution}
are possible under suitable flatness assumptions played
against a non vanishing  Weyl tensor
at a maximum point of $K$, thereby recovering the results of \cite{Rauzy_Multiplicity}.
On the other hand for generic functions $K$ in higher dimensions
the relevant arguments for direct minimization of $J$ or $I$
cannot be applied, as we will discuss in \cite{MZ2}.

Finally we wish to thank Prof. Daniele Bartolucci
for bringing this problem to our attention during our time
at the University of Rome "Tor Vergata".

\section{Preliminaries}\label{sec_preliminaries}
We start with providing the fundamental properties of the flow
generated by \eqref{flow_for_J},
whose short time existence is standard, cf. \cite{F},
provided $k_{u_{0}},r_{u_{0}}<0$ for an initial data, which we assume.
\begin{proposition}\label{prop_flow_properties}
For a positive flow line
$$
u:[0,T)\times C^{\infty}(M,\R_{+})\longrightarrow C^{\infty}(M,\R_{+}):(t,u_{0})\longrightarrow u
\; \text{ with } \;
k_{u_{0}},r_{u_{0}}<0
$$
generated by \eqref{flow_for_J}
and satisfying
$
k,r<0 \; \text{ on } \; [0,T)
$
there holds for all $0\leq t <T$
\begin{enumerate}[label=(\roman*)]
\item
conservation of the volume, i.e. \;
$\partial_{t}\mu_{g_{u}}=\partial_{t}\int_M u^{\frac{2n}{n-2}}d\mu_{g_0}=0.$
\item
non growth of $J$, precisely
$$
\partial_{t}J(u)
=
-\frac{2^*}{(-r)^{\frac{n}{n-2}}}|\delta J|^2(u)
\leq
-\frac{(-r)^{\frac{n}{n-2}}}{2^{*}S^{2}\|u\|_{L^{2^*}}^{\frac{4}{n-2}}}\vert \partial J \vert^{2}(u)
$$
where
$
|\delta J|^2(u)
=
\int_M|\frac{-k}{-r}R-K|^2u^{\frac{2n}{n-2}}d\mu_{g_0}
$
and $S$ is the Sobolev constant.
\item
preservation of positivity, precisely
$$\min\{ 1/C,u_{min}(0)\}\leq u(t)\leq u_{max}(0)e^{Ct},$$
where $C=C(\sup_{[0,t]}(\frac{-k}{-r}+\frac{-r}{-k}))$.
\end{enumerate}
\end{proposition}

\begin{proof}
Property (i) is easy to check by direct computation, as is
\begin{eqnarray*}
\partial_{t} J(u)
&=&
-
\frac{2^{*}}{(-r)^{\frac{n}{n-2}}}
\int_M|\frac{-k}{-r}R-K|^2u^{\frac{2n}{n-2}}d\mu_{g_0}.
\end{eqnarray*}
Moreover
\begin{equation*}
\begin{split}
|\partial J|(u)
= \; &
\sup_{\Vert \varphi\Vert _{H^1(M)}\leq 1}
\int_M \partial J(u)\cdot\varphi d\mu_{g_0}
\\
= \; &
\frac{2^{*}}{(-r)^{\frac{n}{n-2}}}
\sup_{\Vert \varphi\Vert _{H^1(M)}\leq1}
\int_M(\frac{-k}{-r}R-K)u\varphi u^{\frac{4}{n-2}}d\mu_{g_0}.
\end{split}
\end{equation*}
Denote
$dw=u^{\frac{4}{n-2}}d\mu_{g_0}$
with corresponding ${L_w^2}$-norm
$$
\Vert \varphi\Vert^2 _{L_w^2}
=
\int_M\varphi^2dw
=
\int_M\varphi^2u^{\frac{4}{n-2}}d\mu_{g_{0}}.
$$
Then for any $\varphi\in H^1(M)$ by H\"older's inequality we have
$$
\Vert \varphi\Vert _{L_w^2}^2
\leq
(\int_M\varphi^\frac{2n}{n-2}d\mu_{g_0})^{\frac{n-2}{n}}
(\int_Mu^{\frac{2n}{n-2}}d\mu_{g_{0}})^{\frac{2}{n}}
\leq S^2\|u\|_{L^{2^*}}^{\frac{4}{n-2}}
\Vert\varphi\Vert _{H^1(M)}^2,
$$
whence by $L^{2}$-duality
\begin{equation*}
\begin{split}
|\partial J|(u)
\leq \; &
\frac{2^{*}}{(-r)^{\frac{n}{n-2}}}
\sup_{\Vert \varphi\Vert _{L_w^2(M)}\leq S\|u\|_{L^{2^*}}^{\frac{2}{n-2}}}
\int_M(\frac{-k}{-r}R-K)u\varphi u^{\frac{4}{n-2}}d\mu_{g_0}
\\
\leq \; &
\frac{2^{*}S\|u\|_{L^{2^*}}^{\frac{2}{n-2}}}{(-r)^{\frac{n}{n-2}}}
\int_M(\frac{-k}{-r}R-K)u\cdot \frac{(\frac{-k}{-r}R-K)u}{\Vert (\frac{-k}{-r}R-K)u\Vert _{L_w^2}}dw
\\
= \; &
\frac{2^{*}S\|u\|_{L^{2^*}}^{\frac{2}{n-2}}}{(-r)^{\frac{n}{n-2}}}
\Vert (\frac{-k}{-r}R-K)u\Vert _{L_w^2},
\end{split}
\end{equation*}
where $S$ is the Sobolev constant and (ii) is immediate. Recalling finally \eqref{flow_for_J},
the lower bound in (iii) follows from the maximum principle, while the upper one is due
to Gronwall's lemma.
\end{proof}
Note, that we will ensure from an A-B-inequality, that a priori
$\vert k \vert$ and $\vert r \vert $
are uniformly bounded away from
zero and infinity, at least on energy sublevels,
and then long time existence follows from the next two lemmata,
since we already know, that $\inf u \longrightarrow 0$ is impossible,
cf. \cite{Amacha_Regbaoui,Mayer_Scalar_Curvature_Flow}.
\begin{lemma}\label{logineq}
For any $1\leq p\leq \frac{n^2}{2(n-2)}$ we have
$$
\int_M|\frac{-k}{-r}R-K|^{p}d\mu_{g}
\leq
e^{\omega e^{\omega T}},
$$
where $\omega\geq 1$ is bounded against
$\sup_{0\leq t\leq T}\frac{1+\vert r \vert }{\vert k \vert }$.
\end{lemma}
\begin{proof}
Using (\ref{flow_for_J}) by direct calculation, we have
\begin{equation*}
\begin{split}
\partial_t & \int_M|\frac{-k}{-r}R-K|^pd\mu_g \\
= \; & -\frac{4(p-1)}{p}c_n(\frac{-k}{-r})\int_M|\nabla_g|\frac{-k}{-r}R-K|^{\frac{p}{2}}|^2d\mu_g \\
&
+
\frac{4p-2n}{n-2}\int_M|\frac{-k}{-r}R-K|^{p}(\frac{-k}{-r}R-K)d\mu_g
+\frac{4p}{n-2}\int_M|\frac{-k}{-r}R-K|^pKd\mu_g
\\
&
-
\frac{2p}{-k}\int_M|\frac{-k}{-r}R-K|^2d\mu_g\cdot\int_M|\frac{-k}{-r}R-K|^pd\mu_g
\\
&
-
\frac{2p}{-k}\int_M|\frac{-k}{-r}R-K|^2d\mu_g\cdot
\int_M|\frac{-k}{-r}R-K|^{p-2}(\frac{-k}{-r}R-K)Kd\mu_g
\\
&
+
\frac{4p}{n-2}\cdot\frac1{-k}\int_M(\frac{-k}{-r}R-K)Kd\mu_g\cdot\int_M|\frac{-k}{-r}R-K|^pd\mu_g
\\
&
+
\frac{4p}{n-2}\cdot\frac1{-k}\int_M(\frac{-k}{-r}R-K)Kd\mu_g
\cdot
\int_M|\frac{-k}{-r}R-K|^{p-2}(\frac{-k}{-r}R-K)Kd\mu_g.
\end{split}
\end{equation*}
Applying the Sobolev and  H\"older inequalities, we then estimated
\begin{equation*}
\begin{split}
\partial_{t} &  \int_M|\frac{-k}{-r}R -K|^pd\mu_g
+
\omega(\int_M|\frac{-k}{-r}R-K|^{\frac{pn}{n-2}}d\mu_g)^\frac{n-2}{n}
\\
\leq \;&
\frac{4p-2n}{n-2}\int_M|\frac{-k}{-r}R-K|^{p+1}d\mu_g
+
C\int_M|\frac{-k}{-r}R-K|^pd\mu_g
\\
&
-
\frac{2p}{-k}\int_M|\frac{-k}{-r}R-K|^2d\mu_g\cdot\int_M|\frac{-k}{-r}R-K|^pd\mu_g
\\
&
-
\frac{2p}{-k}\int_M|\frac{-k}{-r}R-K|^2d\mu_g\cdot
\int_M|\frac{-k}{-r}R-K|^{p-2}(\frac{-k}{-r}R-K)Kd\mu_g
\\
&
+
\frac{4p}{n-2}\cdot\frac1{-k}\int_M(\frac{-k}{-r}R-K)Kd\mu_g\cdot\int_M|\frac{-k}{-r}R-K|^pd\mu_g
\\
&
+
\frac{4p}{n-2}\cdot\frac1{-k}\int_M(\frac{-k}{-r}R-K)Kd\mu_g\cdot\int_M|\frac{-k}{-r}R-K|^{p-2}(\frac{-k}{-r}R-K)Kd\mu_g,
\end{split}
\end{equation*}
whence
\begin{equation}\label{logequ}
\begin{split}
\partial_{t}   \int_M|\frac{-k}{-r}R & -K|^pd\mu_g
+
\omega(\int_M|\frac{-k}{-r}R-K|^{\frac{pn}{n-2}}d\mu_g)^\frac{n-2}{n}
\\
\leq \; &
\frac{4p-2n}{n-2}\int_M|\frac{-k}{-r}R-K|^{p+1}d\mu_g
+
\omega\int_M|\frac{-k}{-r}R-K|^pd\mu_g
\\
&
+
\omega\int_M|\frac{-k}{-r}R-K|^{2}d\mu_g
+
\omega.
\end{split}
\end{equation}
Let $p=\frac{n}{2}$. We then estimate in case $n \geq 4$
$$
\partial_{t}\int_M|\frac{-k}{-r}R-K|^{\frac{n}{2}}d\mu_g
\leq
\omega\int_M|\frac{-k}{-r}R-K|^{\frac{n}{2}}d\mu_g+\omega,
$$
whence we obtain the logarithmic type estimate
$$
\int_M|\frac{-k}{-r}R-K|^{\frac{n}{2}}d\mu_g
\leq
(\int_M|\frac{-k}{-r}R-K|^{\frac{n}{2}}d\mu_g|_{t=0}
+
\omega)e^{\omega t},
$$
while for $n=3$ we have
$$
\partial_{t}\int_M|\frac{-k}{-r}R-K|^{\frac{3}{2}}d\mu_g
\leq
\omega\int_M|\frac{-k}{-r}R-K|^{2}d\mu_g+\omega,
$$
whence
\begin{equation*}
\begin{split}
\int_M|\frac{-k}{-r}R-K|^{\frac32}d\mu_g
& -
\int_M|\frac{-k}{-r}R-K|^{\frac32}d\mu_g|_{t=0}
\\
\leq \; &
\omega\int_0^t\int_M|\frac{-k}{-r}R-K|^2d\mu_gdt
+
\omega t\leq \omega J(u_0)+\omega t
\end{split}
\end{equation*}
and therefore
$$\int_M|\frac{-k}{-r}R-K|^{\frac32}d\mu_g\leq \omega t+\omega.$$
In conclusion for any $n\geq 3$ we get
$$\int_M|\frac{-k}{-r}R-K|^{\frac n2}d\mu_g\leq \omega e^{\omega t}.$$
Letting $p=\frac n2$ in (\ref{logequ}) and integrating from $0$ to $t$, we therefore have
\begin{equation*}%\label{logequ2}
\int_0^t(\int_M|\frac{-k}{-r}R-K|^{\frac{n^2}{2(n-2)}}d\mu_g)^{\frac{n-2}{n}}dt\leq \omega e^{\omega t}+\omega.
\end{equation*}
Returning to (\ref{logequ}) and applying the H\"older's
and Young's inequality to the term
$$\int_M|\frac{-k}{-r}R-K|^{p+1}d\mu_g,$$
we obtain
\begin{equation*}
\begin{split}
\partial_{t}\int_M|\frac{-k}{-r}R   & \; -K|^pd\mu_g
\leq
\omega(\int_M|\frac{-k}{-r}R-K|^pd\mu_g)^{\frac{2p-n+2}{2p-n}} \\
&  +
\omega\int_M|\frac{-k}{-r}R-K|^pd\mu_g+\omega\int_M|\frac{-k}{-r}R-K|^2d\mu_g+\omega.
\end{split}
\end{equation*}
Taking
$p=\frac{n^2}{2(n-2)}$
and setting
$y=\int_M|\frac{-k}{-r}R-K|^{\frac{n^2}{2(n-2)}}d\mu_g$
we then have
$$\partial_{t}y\leq \omega y^{1+\frac{n-2}{n}}+\omega y+\omega$$
and therefore
$$\log\frac{y+\omega}{(y+\omega)|_{t=0}}\leq \omega \int_0^ty^{\frac{n-2}{n}}dt+\omega t,$$
which implies
$$\int_M|\frac{-k}{-r}R-K|^{\frac{n^2}{2(n-2)}}d\mu_g\leq \omega e^{\omega e^{\omega t}}.$$
The assertion follows.
\end{proof}

\begin{lemma}\label{lem_long_time_existence}
A flow line, for which $\frac{1+\vert r \vert }{\vert k \vert }$ is upper bounded, exists for all time.
\end{lemma}
\begin{proof}
By Lemma \ref{logineq} we have
$$\int_M|R|^pd\mu_g\leq e^{\omega e^{\omega T}}$$
for any $1\leq p\leq \frac{n^2}{2(n-2)}$.
Thus, denoting by $\tilde{\omega}\geq 1$ any quantity bounded against
$$
\sup_{0\leq t \leq T}
(
\Vert u \Vert_{L^{\infty}}
+
\Vert 1/u \Vert_{L^{\infty}}
+
\frac{1+\vert r \vert }{\vert k \vert }
),
$$
we deduce
$$
\int_M|\Delta_{g_0}u|^pd\mu_g\leq e^{\tilde{\omega}e^{\tilde{\omega}T}},
$$
since $u$ is time-dependently bounded.
Then Morrey's inequality shows
$$|u(t,x_1)-u(t,x_2))|\leq C(\tilde{\omega},T)d(x_1, x_2)^\alpha$$
for $x_1, x_2\in M$, $t\in [0, T)$
and
$$\alpha=2-\frac np , \;\frac n2<p<\min\{\frac{n^2}{2(n-2)}, n\}.$$
Moreover from Lemma \ref{logineq}
$$\int_M|\partial_{t} u|^pd\mu_{g_0}\leq C(\tilde{\omega},T).$$
We then obtain for $0<t_{1}-t_{2}<1$
\begin{eqnarray*}
|u(t_{1},x)-u(t_{2},x)|
&=&
\frac1{|B_{\sqrt{t_{1}-t_{2}}}(x)|}\int_{B_{\sqrt{t_{1}-t_{2}}}(x)}|u(t_{1},x)-u(t_{2},x)|d\mu_{g_0}(y)
\\
&=&
\frac1{|B_{\sqrt{t_{1}-t_{2}}}(x)|}\int_{B_{\sqrt{t_{1}-t_{2}}}(x)}|u(t_{1},x)-u(t_{1},y)|d\mu_{g_0}(y)
\\
&&
+
\frac1{|B_{\sqrt{t_{1}-t_{2}}}(x)|}\int_{B_{\sqrt{t_{1}-t_{2}}}(x)}|u(t_{1},y)-u(t_{2},y)|d\mu_{g_0}(y)
\\
&&
+
\frac1{|B_{\sqrt{t_{1}-t_{2}}}(x)|}
\int_{B_{\sqrt{t_{1}-t_{2}}}(x)}|u(t_{2},y)-u(t_{2},x)|d\mu_{g_0}(y)
\\
&=&
I_1+I_2+I_3.
\end{eqnarray*}
The first term on the right hand of the above equality yields to
$$
I_1
\leq
C(\tilde{\omega},T)|t_{1}-t_{2}|^{-\frac 2n}\int_{B_{\sqrt{t_{1}-t_{2}}}(x)}|x-y|^\alpha d\mu_{g_0}(y)
\leq
C(\tilde{\omega},T)|t_{1}-t_{2}|^{\frac \alpha 2}
$$
and similarly
$I_3\leq C(\tilde{\omega},T)|t_{1}-t_{2}|^{\frac \alpha 2}$.
We finally estimate
\begin{eqnarray*}
I_2
&\leq &
C|t_{1}-t_{2}|^{-\frac n2}
\sup_{t_{2}\leq t\leq t_{1}}
\int_{B_{\sqrt{t_{1}-t_{2}}}(x)}|\frac{\partial u}{\partial t}||t_{1}-t_{2}|d\mu_{g_0}(y)
\\
&\leq &
C|t_{1}-t_{2}|^{-\frac n2+1}
\sup_{t_{2}\leq t\leq t_{1}}(\int_{B_{\sqrt{t_{1}-t_{2}}}(x)}|\frac{\partial u}{\partial t}|^pd\mu_{g_0})^{\frac1p}
(\int_{B_{\sqrt{t_{1}-t_{2}}}(x)}\mathbb{1})^{1-\frac1p}
\\
&\leq &
C(\tilde{\omega},T)|t_{1}-t_{2}|^{1-\frac{n}{2p}}
=
C(\tilde{\omega},T)|t_{1}-t_{2}|^{\frac\alpha2},
\end{eqnarray*}
and long time existence follows immediately.
\end{proof}

\section{Existence}%\label{sec_existence}
We essentially show Theorem \ref{thm_minimize_J},
whose proof is found at this section's end.
\subsection{The A-B-inequality}
Recalling \eqref{necessary_condition_nu_1_positive},
we start with proving, that the A-B-conditions (i) and (ii) below
imply some A-B-inequality \eqref{AB}.
\begin{proposition}\label{prop_A_B_inequality_from_A_B_conditions}
There exists $\epsilon>0$ such, that
for any $K\in C^{\infty}(M)$, if
$$
\{ K\geq 0 \} = \Omega_{K} \subset \subset \Omega \subset \subset D
$$
with smooth $\Omega,D\subset M$ and
\begin{enumerate}[label=(\roman*)]
\item
$
\sup_M K
<
\epsilon
[
\; dist^{2\frac{n-1}{n-2}}(\partial \Omega,\partial D)
(\frac{\nu_{1}(D)}{\nu_{1}(D)+1})^{\frac{n}{n-2}}
\;
]
\inf_{M\setminus \Omega} (-K)
$
\item $\nu_{1}(D)=\nu_{1}(L_{g_{0}},D)>0$,
\end{enumerate}
then for some constants $A,B>0$ there holds
\begin{equation}\label{AB}
\Vert u \Vert_{H^{1}}^{2}
\leq
Ar
+
B\vert k \vert^{\frac{n-2}{n}}
\; \text{ for all } \;
u\in H^{1}(M).
\end{equation}
\end{proposition}
We say, that an A-B-inequality holds globally, if \eqref{AB} is satisfied, and likewise holding on $X$ means, that \eqref{AB} holds on $X$ instead of $H^1(M)$.

\begin{proof}
Recalling \eqref{r_and_k_definitions} and by rescaling we have to show
\begin{equation*}
\int_ML_{g_0}uud\mu_{g_0}
+
B\vert \int_MKu^{\frac{2n}{n-2}}d\mu_{g_0}\vert^{\frac{n-2}{n}}
\geq
\epsilon_0
\; \text{ on } \;
\{ \Vert \cdot \Vert_{H^{1}} =1\}.
\end{equation*}
Let $D \subset M$ satisfying
$$
\Omega_K\subset\subset\Omega\subset\subset D
\; \text{ and } \;
\nu_{1}(D)=\nu_{1}(L_{g_{0}},D)>0
$$
and choose a cut off function
$\eta\in C_0^\infty(D)$
with
$0\leq\eta\leq 1$ and $\eta\equiv1$
in
$\Omega$ and $|\nabla\eta|\leq C/d$, where
$d={\rm dist}(\partial \Omega, \partial D)$. We then decompose
$$u=\eta u+(1-\eta)u,\; \eta u \in H^1_{0}(D)$$
and observe, that from $\nu_{1}(D)>0$
$$
c_{n}\int_{M} \vert \nabla (\eta u) \vert^{2}d\mu_{g_{0}}-\int_{M} \vert \eta u \vert^{2}d\mu_{g_{0}}
=
\langle L_{g_{0}}(\eta u),\eta u \rangle
\geq
\nu_{1}(D)\int_{M} \vert \eta u \vert^{2}d\mu_{g_{0}} ,
$$
whence
$
c_{n}\int_{M} \vert \nabla (\eta u) \vert^{2}d\mu_{g_{0}}
\geq
(\nu_{1}(D)+1)\int_{M} \vert \eta u \vert^{2}d\mu_{g_{0}}
$
and therefore
$$
\langle L_{g_{0}}(\eta u),\eta u\rangle
\geq
\frac{c_{n}\nu_{1}(D)}{ \nu_{1}(D)+1}\int \vert \nabla (\eta u) \vert^{2}d\mu_{g_{0}} .
$$
As a consequence
\begin{equation}\label{L_g_0_estimate_to_get_A_B_inequality}
\begin{split}
\langle L_{g_0}u, u\rangle
= \; &
\int_ML_{g_0}(\eta u)(\eta u)d\mu_{g_0}
+
2\int_ML_{g_0}(\eta u)\big((1-\eta)u\big)d\mu_{g_0} \\
& +
\int_ML_{g_0}\big((1-\eta) u\big)\big((1-\eta)u\big)d\mu_{g_0} \\
\geq \; &
\frac{c_n\nu_{1}(D)}{\nu_{1}(D)+1}\int_M|\nabla(\eta u)|^2d\mu_{g_0}
+
c_n\int_M(1-\eta^2)|\nabla u|^2d\mu_{g_0}
\\
&
-
2c_n\int_M u\nabla\eta\cdot\nabla (\eta u)d\mu_{g_0}
+
c_n\int_M|\nabla \eta|^2u^2d\mu_{g_0}
\\
&
-\int_M(1-\eta^2)u^2d\mu_{g_0}.
\end{split}
\end{equation}
From H\"older's inequality we then have
\begin{equation*}
\begin{split}
\bigg|\int_M u\nabla\eta      &        \cdot\nabla(\eta u)d\mu_{g_0}\bigg|
\leq
\frac{C}{d}\bigg(\int_{D\setminus\Omega}u^2d\mu_{g_0}\bigg)^{\frac12}
\bigg(\int_{D\setminus\Omega}|\nabla (\eta u)|^2d\mu_{g_0}\bigg)^{\frac{1}{2}}
\\
\leq \; &
\frac{C}{d}\bigg(\int_{D\setminus\Omega}u^{\frac{2n}{n-2}}d\mu_{g_0}\bigg)^{\frac{n-2}{2n}}
\bigg(\int_{D\setminus\Omega}d\mu_{g_0}\bigg)^{\frac1n}
\bigg(\int_{D\setminus\Omega}|\nabla (\eta u)|^2d\mu_{g_0}\bigg)^{\frac{1}{2}}
\\
\leq \; &
\frac{C|D\setminus\Omega|^{\frac1n}}{d}
\bigg(\int_{D\setminus\Omega}u^{\frac{2n}{n-2}}d\mu_{g_0}\bigg)^{\frac{n-2}{2n}}
\bigg(\int_{D\setminus\Omega}|\nabla (\eta u)|^2d\mu_{g_0}\bigg)^{\frac{1}{2}}
\end{split}
\end{equation*}
and with $C_{3}>0$
\begin{equation*}
\begin{split}
\bigg|\int_{M}(1            & \;      -\eta^2)u^2d\mu_{g_0}\bigg|
\leq
\int_{M\setminus\Omega}u^2d\mu_{g_0} \\
\leq \; &
\bigg(\int_{M\setminus\Omega}u^{\frac{2n}{n-2}}d\mu_{g_0}\bigg)^{\frac{n-2}{n}}
\bigg(\int_{M\setminus\Omega}d\mu_{g_0}\bigg)^{\frac2n}
\leq
C_3\bigg(\int_{M\setminus\Omega}u^{\frac{2n}{n-2}}d\mu_{g_0}\bigg)^{\frac{n-2}{n}}.
\end{split}
\end{equation*}
Plugging these into \eqref{L_g_0_estimate_to_get_A_B_inequality},
whose second last summand is positive anyway, we obtain from Young's inequality
with a constant $C_2>0$ and for any $0<a<1$
\begin{equation}\label{L_g_0_estimate}
\begin{split}
\langle L_{g_0}u   &    ,u\rangle
\geq
\frac{c_n\nu_{1}(D)}{\nu_{1}(D)+1}
\int_M|\nabla(\eta u)|^2d\mu_{g_0}
\\
& -
\frac{2 C c_n|D\setminus\Omega|^{\frac{1}{n}} }{d}
\bigg(\int_{D\setminus\Omega}u^{\frac{2n}{n-2}}d\mu_{g_0}\bigg)^{\frac{n-2}{2n}}
\bigg(\int_{D\setminus\Omega}|\nabla (\eta u)|^2d\mu_{g_0}\bigg)^{\frac{1}{2}}
\\
&
-
C_3\bigg(\int_{M\setminus\Omega}u^{\frac{2n}{n-2}}d\mu_{g_0}\bigg)^{\frac{n-2}{n}}
\\
\geq \; &
\frac{c_n(1-a)\nu_{1}(D)}{\nu_{1}(D)+1}\int_M|\nabla (\eta u)|^2d\mu_{g_0}
\\
&
-
\bigg(
C_{2}
\frac{\nu_{1}(D)+1}{\nu_{1}(D)}
\frac{|D\setminus\Omega|^{\frac{2}{n}} }{ad^2}+C_3\bigg)
\bigg(\int_{M\setminus\Omega}u^{\frac{2n}{n-2}}d\mu_{g_0}\bigg)^{\frac{n-2}{n}}.
\end{split}
\end{equation}
On the other hand, recalling $\{ K\geq 0 \}=\Omega_{K} \subset \subset \Omega $,
\begin{equation}\label{int_K_estimate_on_Omega}
\begin{split}
\vert \int_M K    &      u^{\frac{2n}{n-2}}d\mu_{g_0}\vert^{\frac{n-2}{n}}
\geq \;
\vert
\int_{\Omega^{c}_{K}} Ku^{\frac{2n}{n-2}}d\mu_{g_0}
\vert^{\frac{n-2}{n}}
-
\vert \int_{\Omega_{K}}Ku^{\frac{2n}{n-2}}d\mu_{g_0}\vert^{\frac{n-2}{n}}
\\
\geq \; &
\inf_{M\setminus\Omega}|K|^{\frac{n-2}{n}}
\vert\int_{M\setminus\Omega}u^{\frac{2n}{n-2}}d\mu_{g_0}\vert^{\frac{n-2}{n}}
-
\sup_M K^{\frac{n-2}{n}}
\vert\int_\Omega u^{\frac{2n}{n-2}}d\mu_{g_0}\vert^{\frac{n-2}{n}}.
\end{split}
\end{equation}
Thus combining \eqref{L_g_0_estimate} with \eqref{int_K_estimate_on_Omega} via
\begin{equation*}
\begin{split}
\langle L_{g_{0}}u,u\rangle
& +
B\vert \int_{M} Ku^{\frac{2n}{n-2}}d\mu_{g_{0}}\vert^{\frac{n-2}{n}}
\geq
\frac{c_n(1-a)\nu_{1}(D)}{\nu_{1}(D)+1}\int_M|\nabla (\eta u)|^2d\mu_{g_0}
\\
&
+
(
B\inf_{M\setminus \Omega}\vert K \vert^{\frac{n-2}{n}}
-
C_{2}
\frac{\nu_{1}(D)+1}{\nu_{1}(D)}
\frac{|D\setminus\Omega|^{\frac{2}{n}} }{ad^2}
-C_3
)
\Vert u \Vert_{L^{\frac{2n}{n-2}}(M\setminus \Omega )}^{2} \\
& -
B\sup_M K^{\frac{n-2}{n}}\Vert u \Vert_{L^{\frac{2n}{n-2}}(\Omega)}^{2}.
\end{split}
\end{equation*}
Choosing $a=\frac{1}{2}$, denoting by $S$ the Sobolev constant, supposing
\begin{enumerate}[label=(\roman*)]
\item[1)]
$
B^{\frac{n}{n-2}}
>
2^{\frac{n}{n-2}}(2
C_{2}
\frac{\nu_{1}(D)+1}{\nu_{1}(D)}
\frac{|D\setminus\Omega|^{\frac{2}{n}} }{d^2}
+
C_3
)^{\frac{n}{n-2}}
/
\inf_{M\setminus \Omega}\vert K \vert
$
\item[2)] $\sup_M K<C_{1}/B^{\frac{n}{n-2}}$ with
$C_{1}^{\frac{n-2}{n}}S^2<\frac{c_n\nu_{1}(D)}{4(\nu_{1}(D)+1)}$
,
\end{enumerate}
and recalling $\Omega \subset\subset D$,	
we then find from the Sobolev embedding on $H^{1}(M)$
\begin{equation}\label{L_g_0_u_u_lower_bound_against_L_2_start}
\begin{split}
\langle L_{g_{0}}u,u\rangle
& +
B\vert \int_{M} Ku^{\frac{2n}{n-2}}d\mu_{g_{0}}\vert^{\frac{n-2}{n}}
\geq
\frac{c_n\nu_{1}(D)}{4(\nu_{1}(D)+1)}\int_M|\nabla (\eta u)|^2d\mu_{g_0}
\\
&
+
(
C_{2}
\frac{\nu_{1}(D)+1}{\nu_{1}(D)}
\frac{|D\setminus\Omega|^{\frac{2}{n}} }{2d^2}
+
C_3
)
\Vert u \Vert_{L^{\frac{2n}{n-2}}(M\setminus \Omega )}^{2}
\\
\geq \; &
C(\Omega,D)\Vert u \Vert_{L^{\frac{2n}{n-2}}}^{2}
.
\end{split}
\end{equation}
Adding some $\varepsilon \langle L_{g_{0}}u,u \rangle$
with $\varepsilon>0$ sufficiently small to both sides of
\eqref{L_g_0_u_u_lower_bound_against_L_2_start},
estimate \eqref{AB} readily follows under 1) and 2) above,
which both hold true, if
\begin{equation*}
\begin{split}
\sup_M K
< \; &
\frac{\epsilon \inf_{M\setminus \Omega} (-K)}
{
(
\frac{\nu_{1}(D)+1}{\nu_{1}(D)}\frac{|D\setminus\Omega|^{\frac{2}{n}}}{d^{2}}
+
1
)^{\frac{n}{n-2}}
}
\end{split}
\end{equation*}
for some universal $\epsilon>0$.
Recalling $d=dist(\partial \Omega,\partial D)$, the assertion follows.
\end{proof}

\begin{remark}
In view of the non existence results, discussed in Section \ref{Sec_Introduction},
condition (i) in Proposition \ref{prop_A_B_inequality_from_A_B_conditions}
displays the correct behaviour in that $\sup (K,0)\longrightarrow 0$, as
$\inf(K,0)\longrightarrow 0$ or $\nu_{1}(\Omega_{K})\longrightarrow 0$.
Also note, that
$\nu_{1}(\Omega_{K})\longrightarrow \infty$
does not really relax the smallness assumption.
\end{remark}

\begin{lemma}\label{lem_implications_of_A_B_inequality_on_J}
If an A-B-inequality holds on $X$, then
\begin{enumerate}[label=(\roman*)]
\item 	
$\inf_{X}(-k),\inf_{X}J >0$ and $\sup_{X}(-k),\sup_{X}(-r) <\infty$
\item $\inf_{X}(-r)>0$  on any sublevel $\{ J\leq L \} $.
\end{enumerate}
In particular $\Vert \cdot \Vert_{H^{1}}$ is uniformly bounded on $X$.
\end{lemma}
\begin{proof}
Since $\Vert \cdot \Vert_{L^{\frac{2n}{n-2}}}=1$ on $X$, we have
\begin{enumerate}[label=(\roman*)]
\item[1)]
$
-k
\leq
\Vert K \Vert_{L^{\infty}}\int u^{\frac{2n}{n-2}}d\mu_{g_{0}}
=
\Vert K \Vert_{L^{\infty}}
$
\item[2)]
$
0<
-r
=
-\int (c_{n}\vert \nabla u\vert^{2}-u^{2})d\mu_{g_{0}}
\leq
\int u^{2}d\mu_{g_{0}}
\leq
C_{0}
\Vert u \Vert^{2}_{L^{\frac{2n}{n-2}}}=
C_{0}
$ 	
\end{enumerate}
and may from the Sobolev embedding assume, that
$
1\leq Ar+B\vert k \vert^{\frac{n-2}{n}}.
$
Then $k\longrightarrow 0 \Longrightarrow  2r>1/A$,
while $r<0$ on $X$ by definition. Thus
$$-k\geq c_{0}>0 \; \text{ on  } \; X$$
and therefore
$
J(u)=\frac{-k}{(-r)^{\frac{n}{n-2}}} \geq \frac{c_{0}}{C_{0}^{\frac{n}{n-2}}}
$.
We conclude
$$
L\geq J(u)=\frac{-k}{(-r)^{\frac{n}{n-2}}}
\Longrightarrow
(-r)^{\frac{n}{n-2}}\geq \frac{-k}{L}
\geq
\frac{c_{0}}{L}.
$$
Finally
$\Vert u \Vert_{H^{1}}^{2}\leq C_{1}r+C_{2}\Vert u \Vert_{L^{\frac{2n}{n-2}}}^{2}$ by H\"older's inequality.
\end{proof}

\subsection{Solvability by Minimization}\label{section_solvability_by_minimization}
We show Theorem \ref{thm_minimize_J} by direct minimization, using
the Yamabe type flow \eqref{flow_for_J} in order to pass from a minimizing sequence to a  minimizing sequence of solutions.
\begin{proof}[Proof of Theorem \ref{thm_minimize_J}]
Choose a minimizing sequence of initial data $(u_0^i)\subset X$,
$J(u_0^i)\longrightarrow  \inf_X J$,
and for $i\in \N$ fixed the flow line generated by
\eqref{flow_for_J}
$$u^{i}=u^{i}(t,\cdot), \; u(0,\cdot)=u_0^i.$$
Combining Lemmata \ref{lem_implications_of_A_B_inequality_on_J}, \ref{lem_long_time_existence}
and Proposition \ref{prop_flow_properties} we find
\begin{equation*}
\; \exists \; 0<t_{j}^i\longrightarrow \infty
\; : \;
\vert \partial J \vert(u_{t_{j}}^i)\longrightarrow 0,\;
u_{t_{j}}^i=u(t_{j}^i,\cdot).
\end{equation*}
Moreover, passing to a subsequence, we may assume, that
$$
r_{u_{t_{j}}^i}\longrightarrow r_{\infty}^i<0
\; \text{ and } \;
k_{u_{t_{j}}^i}\longrightarrow k_{\infty}^i<0,
$$
and, since
$\sup_{u\in X}\Vert u \Vert_{H^{1}}<C$ on $X$,
as $r<0$ and $\Vert \cdot \Vert_{L^{\frac{2n}{n-2}}}=1$ thereon,
that
\begin{equation}\label{weak_limits_of_flow_lines}
u^i_{t_{j}}\xrightharpoondown{\quad} u_\infty^i \;\text{ weakly in }\;  H^{1}(M)
\; \text{ and } \;
\sup_{i,j}\Vert u^{i}_{t_{j}} \Vert_{H^{1}}<C.
\end{equation}
In particular $u_\infty^i>0$ from (iii) in Proposition \ref{prop_flow_properties} and is a weak solution of
$$
\frac{-k_{\infty}^i}{-r_{\infty}^i} L_{g_0}u_\infty^i=K(u_\infty^i)^{\frac{n+2}{n-2}},
$$
which implies
\begin{equation}\label{r_infty_k_infty_indentity}
\frac{r_{u_{\infty}^{i}}}{k_{u_{\infty}^{i}}}
=
\frac{r_{\infty}^i}{k_{\infty}^i}.
\end{equation}
As a consequence of \eqref{r_infty_k_infty_indentity}
and the weak lower semicontinuity of the norm $\Vert \cdot \Vert_{H^{1}}$.
\begin{equation}\label{energy_of_weak_limit}
\begin{split}
J(u_{0}^{i})
\geq
\; &\lim_{j \to \infty} J(u_{t_{j}}^{i})
=
\lim_{j\to \infty}
\frac{-k_{u_{t_{j}}^{i}}}{(-r_{u_{t_{j}}^{i}})^{\frac{n}{n-2}}}
=
\lim_{j \to \infty}
(\frac{-k_{u_{t_{j}}^{i}}}{-r_{u_{t_{j}}^{i}}} (-r_{u_{t_{j}}^{i}})^{-\frac{2}{n-2}}) \\
\geq \; &
\frac{-k_{\infty}^i}{-r_{\infty}^i}
\liminf_{j\to \infty} (-r_{u_{t_{j}}^{i}})^{-\frac{2}{n-2}}
\geq
\frac{-k_{\infty}^i}{-r_{\infty}^i}
(-r_{u_{\infty}^{i}})^{-\frac{2}{n-2}}
=
\frac{-k_{u_{\infty}^{i}}}{(-r_{u_{\infty}^{i}})^{\frac{n}{n-2}}} \\
= \; &
J(u_{\infty}^{i}),
\end{split}
\end{equation}
and we deduce $J(u_{\infty}^{i})\longrightarrow \inf_{X}J$, i.e. we have found a minimizing sequence of
solutions.
Finally, passing again to a subsequence, if necessary,
from \eqref{weak_limits_of_flow_lines}
there exist $0\leq u_\infty^\infty\in H^1(M)$
as a weak limit of $u_\infty^i>0$ satisfying
$$\frac{-k_{\infty}}{-r_{\infty}} L_{g_0}u_\infty^\infty=K(u_\infty^\infty)^{\frac{n+2}{n-2}}.$$
Then by standard regularity $0\leq u_{\infty}^{\infty}\in C^{\infty}(M)$
and, writing $\partial J(u_{\infty}^{\infty})=0$
as
$$
-
c_{n}\Delta_{g_{0}}u_{\infty}^{\infty}
-
(1+\frac{-r_{\infty}}{-k_{\infty}}
K(u_{\infty}^{\infty})^{\frac{4}{n-2}})u_{\infty}
=
0,
$$
then $u_{\infty}^{\infty}>0$ follows from the weak Harnack inequality.
Finally, repeating the argument of \eqref{energy_of_weak_limit}, we find
$J(u_{\infty}^{\infty})=\inf_{X}J$ and the proof is complete.
\end{proof}

\section{Non Uniqueness}
In this section we study the occurrence of critical points at infinity, cf. \cite{MM3},
associated to variational formulations of $R=K$ and rule them out, as needed,
cf. Remark \ref{rem_critcal_points_at_infinity}.
We shall throughout this section assume
the validity of some global A-B-inequality  \eqref{AB},
which ensures, that
$$\inf_{X} J=\min_{X}J>0,$$
i.e. a lower, positive bound and
the existence of a minimizer $u_{0}\in X$ of $J$, see Theorem \ref{thm_minimize_J}.
We also assume invertibility of $L_{g_{0}}$, a generic property
\cite{GHJL} implying the existence of a Green's function $G_{g_{0}}$,
see Theorems 2.2 and 3.7 in \cite{Shubin_Spectral_Theory},
which due to non positivity of $L_{g_{0}}$ is necessarily sign changing.
Recall, that on $X$ we reduce $J$ by flowing along
$$
\partial_{t}u
=
-(\frac{k}{r}R-K)u
=
\frac{k}{r}u^{-\frac{4}{n-2}}(c_{n}\Delta_{g_{0}} u
+
u
+
\frac{r}{k}Ku^{\frac{n+2}{n-2}})
$$
and that along every flow line we have
$$
c<-k,-r<C,
$$
see Lemma \ref{lem_implications_of_A_B_inequality_on_J}, while
as a consequence of the parabolic maximum principle every flow line
starting strictly positive remains strictly positive,
and thus a zero weak limit can never occur along \eqref{flow_for_J},
see Proposition \ref{prop_flow_properties},
\textit{in contrast to what happens in case of positive Yamabe invariants}.

On the other hand, if $K$ is sign changing, then
\begin{equation*}%\label{def_Y}
\emptyset\neq Y
=
\{ r>0 \} \cap \{ k>0 \}
\cap
\{ \Vert u \Vert_{L^{\frac{2n}{n-2}}}=1 \}\cap \{ u>0 \}
\subset
C^{\infty}(M),
\end{equation*}
as is easily seen from taking
some $0\leq u_{0}\in C^{\infty}$ with
$$supp(u_{0})\subset B_{\varepsilon}(a_{0})\subset \{ K>0 \} .$$
In fact $k_{u_{0}}>0$ and also $r_{u_{0}}>0$ for $0<\varepsilon \ll 1$
sufficiently small,
as the first Dirichlet eigenvalue of $L_{g_{0}}$ on $B_{\varepsilon}(a_{0})$
is positive,
whence $u=\alpha (u_{0}+\delta) \in Y$ for suitable $\alpha >0$ and $0<\delta \ll 1$.
Moreover we may define analogously to $J$ on $X$
the scaling invariant functional
\begin{equation*}%\label{def_I}
I=\frac{r}{k^{\frac{n-2}{n}}} \; \text{ on } \; Y
\end{equation*}
with naturally associated Yamabe type flow
\begin{equation}\label{flow_for_I}
\partial_{t}u=-(R-\frac{r}{k}K)u,
\end{equation}
for which the same arguments as for \eqref{flow_for_J} in Section \ref{sec_preliminaries} apply,
and so the flow generated by \eqref{flow_for_I} exists for all times,
preserves $Y$ and is a pseudo gradient flow for $I$,
provided $c<k,r<C$ remain uniformly bounded away from zero and infinity
along each flow line,
which by the following lemma holds true.
\begin{lemma}\label{lem_implications_of_A_B_inequality_on_I}
If an A-B-inequality holds on $H^1(M)$, then
\begin{enumerate}[label=(\roman*)]
\item 	$\sup_{Y}k< \infty$ and $\inf_{Y}r,\inf_{Y}I>0 $
\item
$\sup_{Y}r <\infty \; \text{ and }  \inf_{Y}k >0$
on any sublevel $\{ I\leq L \}$.
\end{enumerate}
In particular $\Vert \cdot \Vert_{H^{1}}$ is uniformly bounded on any sublevel.
\end{lemma}	
\begin{proof}
By $\Vert \cdot \Vert_{L^{\frac{2n}{n-2}}}=1$ on $Y$, clearly
$k\leq \max K$, and we may assume
\begin{equation}\label{A_B_inequality_for_proof}
1\leq Ar+B\vert k \vert^{\frac{n-2}{n}}
\; \text{ on } \;
H^{1}\cap \{ \Vert \cdot \Vert_{L^{\frac{2n}{n-2}}} = 1\}.
\end{equation}
Moreover suppose, that $\inf_{Y}r=0$. Then there exists
$(u_{m})\subset Y$ with $r_{u_{m}}\longrightarrow  0$
and hence for any $\epsilon>0$
some $u_{0}\in Y$ with $r_{u_{0}}\leq \epsilon$.
Then consider for $\tau \in [0,1]$
$$
u_{\tau}
=
\frac{(1-\tau)u_{0}+\tau\mathbb{1}}{\Vert (1-\tau)u_{0}+\tau\mathbb{1} \Vert_{L^{\frac{2n}{n-2}}}}
\in H^{1}\cap \{ \Vert \cdot \Vert_{L^{\frac{2n}{n-2}}} = 1\},
$$
for which we readily have
$r_{u_{\tau}}\leq \epsilon$,
while from
$k_{0}=k_{u_{0}}>0$ and $k_{1}=k_{\mathbb{1}}<0$ there exists some
$0<t_{0}<1$ with $k_{t_{0}}=0$.
Then for some $0<t_{1}<t_{0}$ we have
$r_{u_{t_{1}}}\leq \epsilon$
and
$0<k_{u_{t_{1}}}\leq \epsilon$,
contradicting \eqref{A_B_inequality_for_proof}
for $0< \epsilon\ll 1$. Thus
$$
\sup_{Y}k<\infty,\inf_{Y}r>0
\; \text{ and in particular } \;
I=\frac{r}{k^{\frac{n-2}{n}}}\geq C,
$$
whence (i) is shown. To see (ii),
for $u\in \{ I\leq L \} $ from (i) we have
\begin{enumerate}[label=(\roman*)]
\item
$
L\geq I=\frac{r}{k^{\frac{n-2}{n}}}
\Longrightarrow
r\leq Lk^{\frac{n-2}{n}}\leq CL
$	
\item
$
L\geq I=\frac{r}{k^{\frac{n-2}{n}}}
\Longrightarrow
k^{\frac{n-2}{n}}\geq \sfrac{r}{L} \geq \sfrac{C}{L}
$.
\end{enumerate}
Finally
$\Vert u \Vert^{2}_{H^{1}}\leq C_{1}r+C_{2}\Vert u \Vert_{L^{\frac{2n}{n-2}}}^{2}$
and the last assertion follows.
\end{proof}

For the same reasons as for $J$
a flow line $u$ for $I$ along \eqref{flow_for_I}
with a strictly positive initial data $u_{0}$
will remain strictly positive and
hence also for $I$ a zero weak limit will never occur.
Be it as it may, clearly $I$ is a valid variational energy just like $J$,
but in contrast to $J$ we cannot
pass to weak limits $u_{t_{k}} \xrightharpoondown{\quad} u_{\infty}$ within $Y$,
since in that passage $r=\int L_{g_{0}}uud\mu_{g_{0}}>0$
as a just \textit{lower} semicontinuous function
may become non positive, i.e. $r_{u_{\infty}}\leq 0$.
On the other hand, if we can exclude blow-up for $I$,
then we will find a minimizer for $I$,
yielding a second solution to the problem besides the minimizer for $J$.
This will show Theorem \ref{thm_second_solution}.

\subsection{Bubbles and Estimates}\label{section_bubbles}
Let us recall the construction of {\em conformal normal coordinates} from \cite{lp}.
Given $a \in M$, one chooses a special conformal metric
\begin{equation*}%\label{conformal_normal_metric}
\begin{split}
g_{a}=u_{a}^{\frac{4}{n-2}}g_{0}
\; \text{ with } \;
u_{a}=1+O(d^{2}_{g_{0}}(a,\cdot)),
\end{split}
\end{equation*}
whose volume element in $ g_{a}$-geodesic normal coordinates coincides with the Euclidean one, see also \cite{gunther}. In particular
\begin{equation*}
\begin{split}
(exp_{a}^{g_{0}})^{-}\circ \exp_{a}^{g_{a}}(x)
=
x+O(\vert x \vert^{3})
\end{split}
\end{equation*}
for the exponential maps centered at $ a $.
We then denote by $r_a$ the geodesic distance from $a$ with respect to the metric $g_a$
just introduced.
With these choices  the expression of the
Green's function $G_{g_{ a }}$  for the conformal
Laplacian $L_{g_{a}}$ with pole at $a \in M$, denoted by
$G_{a}=G_{g_{a}}(a,\cdot)$, simplifies to
\begin{equation}\label{Expansion_Green_Function}
\begin{split}
G_{ a }=\frac{1}{4n(n-1)\omega _{n}}(r^{2-n}_{a}+H_{ a }), \; r_{a}=d_{g_{a}}(a, \cdot)
, \;
H_{ a }=H_{r,a }+H_{s, a },
\end{split}
\end{equation}
where $\omega_{n} = |S^{n-1}|$, cf. Section 6 in \cite{lp}.
Here $H_{r,a }\in C^{2, \alpha}$,
while
\begin{equation}\label{Irregular_Part}
\begin{split}
H_{s,a}
=
O
\begin{pmatrix}
0 & \text{ for }\, n=3
\\
r_{a}^{2}\ln r_{a} & \text{ for }\, n=4
\\
r_{a}& \text{ for }\, n=5
\\
\ln r_{a} & \text{ for }\, n=6
\\
r_{a}^{6-n} & \text{ for }\, n\geq 7
\end{pmatrix}
\end{split}
\end{equation}
and
$H_{s,a}\equiv 0$, if $g_{a}$ is flat around $a$.
In fact the argument in \cite{lp} of a successive polynomial killing of polynomial deficits,
created by the local geometry, is purely local and thus applies here as well.
Proceeding hence as in \cite{lp} we reach
$$L_{g_{a}}\tilde{G}_{a}=\delta_{a}+ \tilde{f}_{a}
\; \text{ with } \;
\tilde{G}_{a}
\; \text{ as in \eqref{Expansion_Green_Function} and } \;
\tilde{f}_{a}\in C^{0,\alpha}$$
and,
solving $L_{g_{a}}\tilde F_{a}=-\tilde f_{a}$ by assumed invertibility of $L_{g_{a}}$,
we find $\tilde F_{a}\in C^{2,\alpha}$ from Schauder estimates and then $G_{a}=\tilde{G}_{a}+\tilde{F}_{a}$.
On $\{ G_{a}>0 \} $ let
for $\lambda>0$
\begin{equation*}%\label{theta_a_lambda}
\theta_{a,\lambda}=u_{ a }(\frac{\lambda}{1+\lambda^{2}\gamma_{n} G_{a}^{\frac{2}{2-n}}})^{\frac{n-2}{2}}
,\;
\gamma_{n}=(4n(n-1)\omega _{n})^{\frac{2}{2-n}},
\end{equation*}
see \cite{MM1} or \cite{Mayer_Scalar_Curvature_Flow}. Extend $\theta_{a,\lambda}=0$ on $\{ G_{a}\leq 0 \} $  and
with a smooth cut-off function
$$
\eta_{a}=\eta(d_{g_{a}}(a,\cdot))
=
\left\{
\begin{matrix*}
1  &  \;\text{on}\;  &  B_{\epsilon}(a) &=&B^{d_{g_{a}}}_{\epsilon}(a) \\
0 &  \;\text{on}\;  &   B_{2\epsilon}(a)^{c}&=&M\setminus  B^{d_{g_{a}}}_{2\epsilon}(a) \}, \\
\end{matrix*}
\right.
$$
where $0<\epsilon\ll 1$ is independent of $a\in M$ and such, that on
$B^{d_{g_{a}}}_{4\epsilon}(a)$ the conformal normal coordinates
from $g_{a}$ are well defined and $G_{g_{a}}>0$, define
\begin{equation}\label{Bubble_Definition}
\begin{split}
\varphi_{a, \lambda }
= &
\eta_{a}\theta_{a,\lambda}.
\end{split}
\end{equation}
Note,  that
$
\gamma_{n}G^{\frac{2}{2-n}}_{ a }(x) =
d_{g_a}^{2}(a,x) + o(d_{g_a}^{2}(a,x))$,
as $x \longrightarrow a$.
\begin{lemma}\label{lem_L_g_0_of_bubble}
If $Cond_{n}$ holds at $a\in M$, then
\begin{equation*}
L_{g_{0}}\varphi_{a, \lambda}
=
4n(n-1)\varphi_{a, \lambda}^{\frac{n+2}{n-2}}
+
O(\frac{\chi_{B_{2\epsilon}(a)\setminus B_{\epsilon}(a)}}{\lambda^{\frac{n-2}{2}}})
+
o_{\frac{1}{\lambda}}(\frac{1}{\lambda^{\frac{n-2}{2}}}) \; \text{ in } \; W^{-1,2}(M).
\end{equation*}
The expansion above persists upon taking the $\lambda \partial_\lambda$ and $\frac{\nabla_{a}}{\lambda}$ derivatives.
\end{lemma}
\begin{proof}
As in Lemma 3.3 in \cite{Mayer_Scalar_Curvature_Flow} and denoting  $r_{a}=d_{g_{a}}(a,\cdot)$ we find
\begin{equation*}
\begin{split}
L_{g_{0}}\theta_{a, \lambda}
= &
4n(n-1)\theta_{a, \lambda}^{\frac{n+2}{n-2}}
+
\frac{u_{a}^{\frac{2}{n-2}}R_{g_{a}}}{\lambda}\theta_{a, \lambda}^{\frac{n}{n-2}}
\\
& -
2nc_{n}
(1+o_{r_{a}}(1))r_{a}^{n-2}((n-1)H_{a}+r_{a}\partial_{r_{a}}H_{a}) \theta_{a, \lambda}^{\frac{n+2}{n-2}}
\end{split}
\end{equation*}
pointwise and thus
\begin{equation*}%\label{L_g_0_varphi_a_lambda_pointwise}
\begin{split}
& L_{g_{0}}  \varphi_{a,\lambda}
= \;
4n(n-1)\varphi_{a,\lambda}^{\frac{n+2}{n-2}}\\
&
-c_{n}\Delta_{g_{0}}\eta_{a}\theta_{a,\lambda}-2c_{n}\langle\nabla \eta_{a},\nabla \theta_{a,\lambda}\rangle_{g_{0}}
+4n(n-1)(\eta_{a}-\eta_{a}^{\frac{n+2}{n-2}})\theta_{a,\lambda}^{\frac{n+2}{n-2}}
\\
&
-
2nc_{n}
(1+o_{r_{a}}(1))r_{a}^{n-2}((n-1)H_{a}+r_{a}\partial_{r_{a}}H_{a}) \eta_a \theta_{a, \lambda}^{\frac{n+2}{n-2}} +
\frac{u_{a}^{\frac{2}{n-2}}R_{g_{a}}}{\lambda}\eta_a\theta_{a, \lambda}^{\frac{n}{n-2}}.
\end{split}
\end{equation*}
Readily the terms of the second line above can be subsumed under some
$$O(\lambda^{\frac{2-n}{2}})\chi_{B_{2\epsilon}(a)\setminus B_{\epsilon}(a)},$$
while those of the third are of order $o(\lambda^{\frac{2-n}{2}})$ in $W^{-1,2}(M)$,
as follows from
\begin{enumerate}[label=(\roman*)]
\item 	$R_{g_{a}}=O(r_{a}^{2})$ and \eqref{Irregular_Part} in case $3\leq n \leq 5$
\item   $H_{s,a}=R_{g_{a}}= 0$ close to $a$ in case of local flatness around $a$ for $n\geq 6$
\end{enumerate}
and simple integral estimates.
\end{proof}
Since we target $\lambda^{\frac{2-n}{2}}$ as the level of precision and
from Lemma \ref{lem_L_g_0_of_bubble}
$$
\int \vert L_{g_{0}} \varphi_{a,\lambda}-4n(n-1)\varphi_{a,\lambda}^{\frac{n+2}{n-2}}\vert \theta_{a,\lambda}d\mu_{g_{0}}		
=
o_{\frac{1}{\lambda}}(\frac{1}{\lambda^{\frac{n-2}{2}}}),
$$

the error terms in Lemma \ref{lem_L_g_0_of_bubble} will be of no concern.

\
 
\noindent {\bf Notation.}
For  points $a_i \in M$  we will denote  by $K_{i},\nabla K_i$ and $\Delta K_i$ for instance
$$
K(a_{i}),\;
\nabla K(a_{i}) = \nabla_{g_0} K(a_{i})
\; \text{ and } \;
\Delta K(a_{i}) = \Delta_{g_0} K(a_{i}).$$
\noindent
For $k,l=1,2,3$ and $ \lambda_{i} >0, \, a _{i}\in M, \,i= 1, \ldots,q$ let
\begin{enumerate}[label=(\roman*)]
 \item
$
\varphi_{i}
=
\varphi_{a_{i}, \lambda_{i}}$ and $(d_{1,i},d_{2,i},d_{3,i})
=
(1,-\lambda_{i}\partial_{\lambda_{i}}, \frac{1}{\lambda_{i}}\nabla_{a_{i}});
$
 \item
$\phi_{1,i}=\varphi_{i}, \;\phi_{2,i}
=
-\lambda_{i} \partial_{\lambda_{i}}\varphi_{i},
\;\phi_{3,i}= \frac{1}{\lambda_{i}} \nabla_{ a _{i}}\varphi_{i}$, so
$
\phi_{k,i}=d_{k,i}\varphi_{i}.
$
\end{enumerate}
Note, that with the above definitions $\phi_{k,i}$ is uniformly bounded in $H^{1}(M)$.
\begin{lemma}\label{lem_interactions}
Let $k,l=1,2,3$ and $i,j = 1, \ldots,q$. Then for
\begin{equation*}
\varepsilon_{i,j}
=
\eta(d_{g_{0}}(a_{i},a_{j}))
(
\frac{\lambda_{j}}{\lambda_{i}}
+
\frac{\lambda_{i}}{\lambda_{j}}
+
\lambda_{i}\lambda_{j}\gamma_{n}G_{g_{0}}^{\frac{2}{2-n}}(a _{i},a _{j})
)^{\frac{2-n}{2}}
\end{equation*}
with a suitable cut-off function
$$
\eta
=
\left\{
\begin{matrix*}[l]
1  &  \;\text{on}\;  &  r<4\epsilon \\
0 &  \;\text{on}\;  &  r\geq 6 \epsilon \\
\end{matrix*}
\right.
$$
and $\epsilon>0$ sufficiently small there holds
\begin{enumerate}[label=(\roman*)]
\item
$
\vert \phi_{k,i}\vert,
\vert \lambda_{i}\partial_{\lambda_{i}}\phi_{k,i}\vert,
\vert \frac{1}{\lambda_{i}}\nabla_{a_{i}} \phi_{k,i}\vert
\leq
C \varphi_{i} + o(\frac{1}{\lambda_{i}^{\frac{n-2}{2}}})$
\item
$
\int \varphi_{i}^{\frac{4}{n-2}}
\phi_{k,i}\phi_{k,i}d\mu_{g_{0}}
=
c_{k}\cdot id
+
O(\frac{1}{\lambda_{i}^{2}})
+
o(\frac{1}{\lambda_{i}^{\frac{n-2}{2}}})
, \;c_{k}>0$
\item
for  $i\neq j$ up to some
$O(\sum_{i\neq j}(\frac{1}{\lambda_{i}^{4}}
+
\varepsilon_{i,j}^{\frac{n+2}{n}}))
+
o(\sum_{i\neq j}\frac{1}{\lambda_{i}^{\frac{n-2}{2}}})$
\begin{equation*}
\int \varphi_{i}^{\frac{n+2}{n-2}}\phi_{k,j}d\mu_{g_{0}}
=
b_{k}d_{k,i}\varepsilon_{i,j}
=
\int \varphi_{i}d_{k,j}\varphi_{j}^{\frac{n+2}{n-2}}  d\mu_{g_{0}}
\end{equation*}
 \item
$
\int \varphi_{i}^{\frac{4}{n-2}}
\phi_{k,i}\phi_{l,i}d\mu_{g_{0}}
=
O(\frac{1}{\lambda_{i}^{2}})+o(\frac{1}{\lambda_{i}^{\frac{n-2}{2}}})$
for $k\neq l$ and for $k=2,3$
$$
\int \varphi_{i}^{\frac{n+2}{n-2}}
\phi_{k,i}d\mu_{g_{0}}
=
O
(
\frac{1}{\lambda_{i}^{4}})
+
o(\frac{1}{\lambda_{i}^{\frac{n-2}{2}}})
$$
 \item
$
\int \varphi_{i}^{\alpha}\varphi_{j}^{\beta} d\mu_{g_{0}}
=
O(\varepsilon_{i,j}^{\beta})
$
for $i\neq j,\;\alpha +\beta=\frac{2n}{n-2}, \; \alpha>\frac{n}{n-2}>\beta\geq 1$
\item
$
\int \varphi_{i}^{\frac{n}{n-2}}\varphi_{j}^{\frac{n}{n-2}} d\mu_{g_{0}}
=
O(\varepsilon^{\frac{n}{n-2}}_{i,j}\ln \varepsilon_{i,j}), \,i\neq j
$
 \item
$
(1, \lambda_{i}\partial_{\lambda_{i}}, \frac{1}{\lambda_{i}}\nabla_{a_{i}})\varepsilon_{i,j}=O(\varepsilon_{i,j})
+
o(\frac{1}{\lambda_{i}^{\frac{n-2}{2}}}+\frac{1}{\lambda_{j}^{\frac{n-2}{2}}})
, \,i\neq j$,
\end{enumerate}
with constants $ b_{k}=\underset{\R^{n}}{\int}\frac{dx}{(1+r^{2})^{\frac{n+2}{2}}}$ for $k = 1, 2, 3$ and
$$
c_{1}=\underset{\R^{n}}{\int}\frac{dx}{(1+r^{2})^{n}},\,
c_{2}=\frac{(n-2)^{2}}{4}\underset{\R^{n}}{\int}\frac{( r^{2}-1)^{2}dx}{(1+r^{2})^{n+2}},\,
c_{3}=\frac{(n-2)^{2}}{n}\underset{\R^{n}}{\int}\frac{r^{2}dx}{(1+r^{2})^{n+2}}.
$$
\end{lemma}
\begin{proof}
First note, that $\varepsilon_{i,j}$ is well defined, as
$G_{g_{0}}(a_{i},a_{j})>0$, if $a_{i},a_{j}$ are close.
Secondly the estimates above are just those of
Lemma 3.5\footnote
{
See also Lemma 3.5 and its proof in the more
detailed version of \cite{Mayer_Scalar_Curvature_Flow}
found at
\url{http://geb.uni-giessen.de/geb/volltexte/2015/11691/}
}
in \cite{Mayer_Scalar_Curvature_Flow} up to some
$$
o(\lambda_{i}^{-\frac{n-2}{2}})
\; \text{ or } \;
o(\lambda_{i}^{-\frac{n-2}{2}}+\lambda_{j}^{-\frac{n-2}{2}})
$$ respectively,
depending on whether we calculate an inter- or selfactions of bubbles.
These errors account for the cut-off functions in \eqref{Bubble_Definition}. Concerning
\begin{enumerate}[label=(\roman*)]
\item 	interactions of two bubbles
$\varphi_{i}$ and $\varphi_{j}$, $i\neq j$, as in e.g. (iii),
these
\begin{enumerate}[label=(\roman*)]
\item[(1)] are zero, if $a_{i},a_{j}$ are far, e.g. $d_{g_{0}}(a_{i},a_{j})>4\epsilon$
\item[(2)] are of order $o(\sfrac{1}{\lambda_{i}^{\frac{n-2}{2}}}+\sfrac{1}{\lambda_{j}^{\frac{n-2}{2}}})$, in case $d_{g_{0}}(a_{i},a_{j})\geq \frac{\epsilon}{4}$
\item[(3)] or else expand as in Lemma 3.5 in \cite{Mayer_Scalar_Curvature_Flow}.
\end{enumerate}
\item the various selfactions of one bubble, e.g. (ii), in view of
Lemma \ref{lem_L_g_0_of_bubble} and
$$\varphi_{a_{i},\lambda_{i}}=O(\lambda_{i}^{-\frac{n-2}{2}}) \; \text{ on } \; B_{\epsilon}(a_{i})^{c}$$
the estimates are up to some $o(\lambda_{i}^{-\frac{n-2}{2}})$
the same as in Lemma 3.5 in \cite{Mayer_Scalar_Curvature_Flow}.
\end{enumerate}
With these hints in mind we leave the details to the reader.
\end{proof}

\subsection{Blow-up Description}
The general characterization of Palais-Smale sequences hold for $I$ and $J$ alike.

\begin{proposition}\label{prop_blow_up_analysis}
Let
$K\in C^{\infty}(M)$ satisfy some A-B-inequality on $H^1(M)$ and
$$(u_{m}) \subset X \; \text{ or } \; \;(u_{m})	\subset \;Y$$
be a sequence, for which
\begin{equation*}
J(u_{m})
=
\frac{-k_{u_{m}}}{(-r_{u_{m}})^{\frac{n}{n-2}}}
\longrightarrow
c>0
\;\text{ and }\;
\partial J(u_{m})\longrightarrow 0
\; \text{ in }\;
W^{-1,2}(M).
\end{equation*}
or
\begin{equation*}
I(u_{m})
=
\frac{r_{u_{m}}}{(k_{u_{m}})^{\frac{n-2}{n}}}
\longrightarrow
c>0
\;\text{ and }\;
\partial I(u_{m})\longrightarrow 0
\; \text{ in }\;
W^{-1,2}(M)
\end{equation*}
respectively. Then  up to a subsequence
$$
r_{u_{m}}=r_{m}\longrightarrow r_{\infty}<0
,\;
k_{u_{m}}=k_{m}\longrightarrow k_{\infty}<0
$$
or
$$
r_{u_{m}}=r_{m}\longrightarrow r_{\infty}> 0
,\;
k_{u_{m}}=k_{m}\longrightarrow k_{\infty}> 0
$$
respectively
and in either case
there exist a solution
$$0\leq u_\infty \in C^{\infty}(M)
\; \text{ to } \;
L_{g_{0}}u_{\infty}=Ku_{\infty}^{\frac{n+2}{n-2}}
\; \text{ with either } \;
u_{\infty}\equiv 0
\; \text{ or } \; 	
u_{\infty}>0,
$$
a sequence $\R_{+}\supset(\alpha_{m})\longrightarrow \alpha_{\infty}$,
$q \in \mathbb{N}_{0}$ and for $i=1,\ldots,q$ sequences
\begin{equation*}\begin{split}
M \supset (a_{i,m})\xlongrightarrow{m \to \infty} a_{i_{\infty}}
\; \text{ and }\;
\R_{+}\supset \lambda_{i,m} \xlongrightarrow{m \to \infty} \infty
\end{split}\end{equation*}
such,  that
$u_{m}
=
\alpha_{m}	u_{\infty }
+
\sum_{i=1}^{q}\alpha_{i,m}\varphi_{a_{i,m},\lambda_{i,m}}+v_{m}$
with
\begin{equation*}
\Vert v_{m} \Vert \longrightarrow 0,\;
\frac{r_{\infty}K(a_{i_{\infty}})\alpha_{i_{\infty}}^{\frac{4}{n-2}}}
{4n(n-1)k_{\infty}}= 1
\; \text{ and, if }\;
u_{\infty} \not \equiv 0,
\; \text{ then } \;
\frac{r_{\infty}\alpha_{\infty}^{\frac{4}{n-2}}}{k_{\infty}}=1.
\end{equation*}
Moreover, if $q>1$, then $(\varepsilon_{i,j})_{m} \xrightarrow{m\to \infty} 0$
for each pair $1\leq i<j\leq q$.
 \end{proposition}
\begin{proof}
This follows as in \cite{MM1}, which does not rely on the positivity of $L_{g_{0}}$, but on being
able to pass $r_{m}\longrightarrow r_{\infty}\neq 0$ and $k\longrightarrow k_{\infty}\neq 0$,
which we may assume here thanks to Lemmata
\ref{lem_implications_of_A_B_inequality_on_J} and \ref{lem_implications_of_A_B_inequality_on_I}.
\end{proof}
Note, that $L_{g_{0}}u_{\infty}=Ku_{\infty}^{\frac{n+2}{n-2}}$ implies
$r_{u_{\infty}}=k_{u_{\infty}}$, whence $r_{u_{\infty}}=0$ or $k_{u_{\infty}}=0$ are impossible
from the presumed validity of an A-B-inequality.
Moreover note, that necessarily $K(a_{i_{\infty}})>0$, that is, if blow-up happens, then on $\{ K>0 \}$.

In order to have a unique representation of a blow-up scenario,
we perform a standard reduction procedure, see \cite{Mayer_Scalar_Curvature_Flow}, as follows.
\begin{definition}%\label{V_p_e}
For $\varepsilon>0,\; q\in \N$ and $ u\in H^{1}(M)$ let

\begin{enumerate}[label=(\roman*),leftmargin=16pt]
\item
$ _{} $ \vspace{-21pt}
\begin{equation*}
\begin{split}
\hspace{-2pt}
A_{u}(u_{\infty,}q, \varepsilon)
 =
\{
&
(
\alpha ,\alpha_{i}, \lambda_{i}, a_{i})\in \R \times \R^{q}\times \R^{q}_{+} \times M^{q} \; : \;
\underset{i\neq j}{\forall\;}\;
\lambda_{i}^{-1}, \lambda_{j}^{-1}, \varepsilon_{i,j},
\\
&
\vert 1-\frac{r \alpha^{\frac{4}{n-2}} }{k} \vert,
\vert 1-\frac{r\alpha_{i}^{\frac{4}{n-2}}K(a_{i})}{4n(n-1)k}\vert,
\Vert u-\alpha u_{\infty}-\alpha^{i}\varphi_{a_{i}, \lambda_{i}}\Vert
\leq\varepsilon
\}
\end{split}
\end{equation*}
 \item
 $
 V(u_{\infty}, q, \varepsilon)
 =
 \{
 u\in W^{1,2}(M)
 \mid
 A_{u}(u_{\infty},q, \varepsilon)\neq \emptyset
 \}
 $.
\end{enumerate}
Here $\|\cdot\| = \|\cdot\|_{D_{g_0}}$ denotes the Sobolev norm induced by $D_{g_0} = L_{g_0}+2$.
\end{definition}
As shown in the Appendix, given $u \in V(u_{\infty},q,\varepsilon)$
and fixed $(\bar{a}_{i},\bar\lambda_{i})$,
by invertibility of the linear operator $L_{g_{0}}$ we may \textit{uniquely} write
\begin{equation}\label{orthogonalities_v_to_u_infty_and_varphi}
u=\bar{\alpha }u_{\infty}+\bar{\alpha}^{i}\varphi_{ \bar{a}_i,\bar{\lambda}_i}+\bar{v}
\; \text{ with } \;
\bar{v} \perp_{L_{g_{0}}}span\{u_{\infty},\varphi_{ \bar{a}_i,\bar{\lambda}_i}\}
\end{equation}
and then minimize along natural projection $\Pi$ induced by \eqref{orthogonalities_v_to_u_infty_and_varphi}, as follows.
\begin{lemma}\label{lem_optimal_choice} %\\
For every $\varepsilon_{0}>0$  there exists $0<\varepsilon_{2}<\varepsilon_{1}<\varepsilon_{0}$ such,
that  for any
$$u\in V(u_{\infty}, q, \varepsilon_{2})$$
the problem
\begin{equation*}\begin{split}
\inf_
{
(\tilde a_{i}, \tilde\lambda_{i})
\in \Pi_{(a_i, \lambda_i)}A_{u}(u_{\infty},q,2\varepsilon_{1})
}
\int u^{\frac{4}{n-2}}\vert u-\tilde{\alpha}u_{\infty} - \tilde{\alpha}^{i}\varphi_{a_{i},\lambda_{i}} \vert^{2}d\mu_{g_{0}}
\end{split}\end{equation*}
admits a unique minimizer $(a_i, \lambda_i)$ with $(\alpha, \alpha_i, a_i, \lambda_i)\in A_{u}(u_{\infty},q,\varepsilon_{1})$. Setting
\begin{equation*}
\begin{split} %\label{eq:v}
\varphi_{i}=\varphi_{a_{i}, \lambda_{i}},
\;
v=u-\alpha u_{\infty} - \alpha^{i}\varphi_{i},
\end{split}
\end{equation*}
we have in addition to $v \perp_{L_{g_{0}}}span\{u_{\infty},\varphi_{ a_i,\lambda_i}\}$ from  \eqref{orthogonalities_v_to_u_infty_and_varphi}, that
\begin{enumerate}[label=(\roman*)]
\item the quantities
$\langle \lambda_{i}\partial_{\lambda_{i}}\varphi_{a_{i},\lambda_{i}},v\rangle_{L_{g_{0}}}$
and
$\int u^{\frac{4}{n-2}}\lambda_{i}\partial_{\lambda_{i}}\varphi_{a_{i},\lambda_{i}}vd\mu_{g_{0}}$
are of order
\begin{equation*}
\begin{split}
O
(
\sum_{i}
\frac{1}{\lambda_{i}^{n-2}}      &       + \sum_{i\neq j}\varepsilon_{i,j}^{2}  +  \Vert v \Vert^{2}
)
+
O(\sum_{i}
\vert
\langle \varphi_{a_{i},\lambda_{i}},\lambda_{i}\partial_{\lambda_{i}}\varphi_{a_{i},\lambda_{i}}\rangle_{L_{g_{0}}} \vert^{2}
) \\
& +
O(\sum_{i}\Vert
\lambda_{i}\partial_{\lambda_{i}}L_{g_{0}}\varphi_{a_{i},\lambda_{i}}
-
4n(n-1)\lambda_{i}\partial_{\lambda_{i}}\varphi_{a_{i},\lambda_{i}}^{\frac{n+2}{n-2}}
\Vert_{L_{g_{0}}^{2}}^{2}
)
\end{split}
\end{equation*}
and, if $Cond_{n}$ holds at all $a_{i}$, of order
$
o_{\varepsilon_{1}}(\sum_{i}\frac{1}{\lambda_{i}^{\frac{n-2}{2}}} )
+
O( \sum_{i\neq j}\varepsilon_{i,j}^{2} + \Vert v \Vert^{2} )
$

\item the quantities
$\langle \frac{\nabla_{a_{i}}}{\lambda_{i}}\varphi_{a_{i},\lambda_{i}},v\rangle_{L_{g_{0}}}$
and
$\int u^{\frac{4}{n-2}}\frac{\nabla_{a_{i}}}{\lambda_{i}}\varphi_{a_{i},\lambda_{i}}vd\mu_{g_{0}}$
are of order
\begin{equation*}
\begin{split}
O
(
\sum_{i}
\frac{1}{\lambda_{i}^{n-2}}       &        + \sum_{i\neq j}\varepsilon_{i,j}^{2}  +  \Vert v \Vert^{2}
)
+
O(\sum_{i}
\vert
\langle \varphi_{a_{i},\lambda_{i}},\frac{\nabla_{a_{i}}}{\lambda_{i}}\varphi_{a_{i},\lambda_{i}}\rangle_{L_{g_{0}}} \vert^{2}
) \\
& +
O(\sum_{i}\Vert
\lambda_{i}\partial_{\lambda_{i}}L_{g_{0}}\varphi_{a_{i},\lambda_{i}}
-
4n(n-1)\lambda_{i}\partial_{\lambda_{i}}\varphi_{a_{i},\lambda_{i}}^{\frac{n+2}{n-2}}
\Vert_{L_{g_{0}}^{2}}^{2}
)
\end{split}
\end{equation*}
and, if $Cond_{n}$ holds at all $a_{i}$, of order
$
o_{\varepsilon_{1}}(\sum_{i}\frac{1}{\lambda_{i}^{\frac{n-2}{2}}} )
+
O( \sum_{i\neq j}\varepsilon_{i,j}^{2} + \Vert v \Vert^{2} ).
$
\end{enumerate}
\end{lemma}
\noindent
The proof, postponed to the Appendix, is technically analogous to the case of a positive Yamabe invariant,
cf. Appendix A in \cite{Bahri_Coron_Three_Dimensional_Sphere} and Proposition 3.10 in \cite{Mayer_Scalar_Curvature_Flow},
where thanks to $L_{g_{0}}>0$ we may  minimize over all variables
$$
\inf_{(\alpha,\alpha_{i},a_{i},\lambda_{i})\in A_{u}(u_{\infty},q,2\varepsilon_{1})}
\Vert u-\alpha u_{\infty}-\alpha^{i}\varphi_{a_{i},\lambda_{i}} \Vert_{L_{g_{0}}}^{2},
$$
which is not feasible in our context.
Here the \textit{linear} variables $(\alpha,\alpha_i)$ are chosen for the sake of the $L_{g_{0}}$-orthogonalities in \eqref{orthogonalities_v_to_u_infty_and_varphi},
since for proper estimates and expansions, which we will require in  \cite{MZ2}, at least $\langle u_{\infty},v \rangle_{L_{g_{0}}}=0$ is indispensable.
Fortunately Lemma \ref{lem_optimal_choice} still provides sufficient \textit{almost}-orthogonalities in $(\lambda_{i},a_{i})$.

\begin{remark}
With these notions at hand the lack of zero weak limit blow-ups along flow lines
and in particular of, as they are referred to, pure critical points at infinity becomes {\normalfont{heuristically}} clear.
\begin{figure}[h!]
\centering
\includegraphics[width=\textwidth]{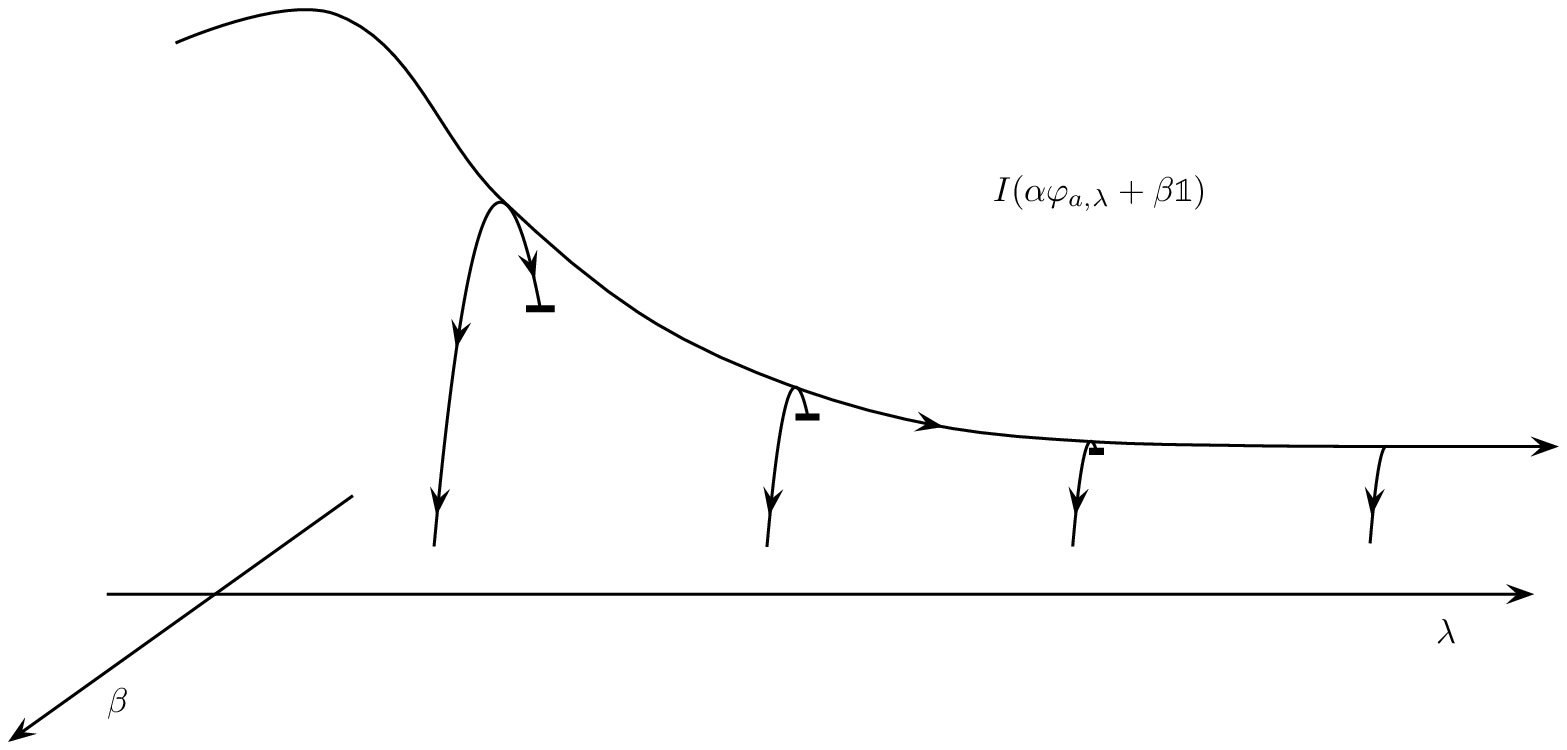}
\caption{}
\label{pic_v_part}
\centering
\end{figure}
Indeed consider $u=\alpha \varphi +v \in Y$ and compute from
Proposition \ref{prop_blow_up_analysis} for instance and up to some $o(\Vert v \Vert^{2})$
\begin{equation*}
\begin{split}
I(u)
= \; &
\frac{r}{k^{\frac{n-2}{n}}}
=
\frac{r_{\varphi}}{k_{\varphi}^{\frac{n-2}{n}}}
+
o_{\frac{1}{\lambda}}(\Vert v \Vert) \\
& +
\frac{1}{k^{\frac{n-2}{n}}}
(
\int L_{g_{0}}vvd\mu_{g_{0}}
-
\frac{n+2}{n-2}
\frac{r\alpha^{\frac{4}{n-2}}}{k}\int K\varphi^{\frac{4}{n-2}}v^{2}d\mu_{g_{0}}
)
\\
= \; &
\frac{r_{\varphi}}{k_{\varphi}^{\frac{n-2}{n}}}
+
o_{\frac{1}{\lambda}}(\Vert v \Vert)
+
\frac{c_{n}}{k^{\frac{n-2}{n}}}
(
\int \frac{L_{g_{0}}vv}{c_{n}}d\mu_{g_{0}}
-
n(n+2)
\int \varphi^{\frac{4}{n-2}}v^{2}d\mu_{g_{0}}
).
\end{split}
\end{equation*}
Thus for $\lambda \gg 1$ the constant functions $\pm \mathbb{1}$ become
principally negative directions, see Figure \ref{pic_v_part}, i.e. directions along which we may decrease energy,
in contrast to the positive Yamabe case, where in a suitable Yamabe metric
$L_{g_{0}}=-c_{n}\Delta_{g_{0}}+1$
instead of, as here,
$L_{g_{0}}=-c_{n}\Delta_{g_{0}}-1$.
That is, in the positive Yamabe case
it is generally opportune to decrease $v$,
while here it is not and we furthermore cannot reasonably
increase $v$ along $-\mathbb{1}$, since then $u>0$ would lose positivity.
Conversely $+\mathbb{1}$ becomes a preferred direction to decrease energy.
On the other hand, if $u=\alpha_{0}u_{\infty} + \alpha \varphi_{a,\lambda} + v$ with $u_{\infty}>0$
these issues clearly do not occur.
\end{remark}

\subsection{Compactness and Existence}\label{sec_compactness_and_existence}
As discussed above, the case $u_{\infty}=0$ does not occur, when flowing,
while non zero weak limit blow-ups may principally occur for $I$ or $J$.
The latter case is ruled out under a smallness assumption on $K$.
\begin{lemma}\label{lem_no_blow_up_for_J}
If an A-B-inequality holds on $H^1(M)$, then the flow generated by
$$
\partial_{t}u=-(\frac{k}{r}R-K)u
$$
is compact on $X$, provided  $ 0\leq \sup_M K $ sufficiently small.
\end{lemma}
\begin{proof}
If blow-up occurs for $J$ along a flow line for a sequence in time,
i.e. for some
$u_{m}=u_{t_{m}}\in X$, then from Proposition \ref{prop_blow_up_analysis}
and Lemma \ref{lem_interactions}
\begin{equation}\label{r_on_X_no_blow_up}
\begin{split}
0>r_{u_{m}}
= & \;
\int L_{g_{0}}u_{m}u_{m}d\mu_{g_{0}}
=
\alpha_{\infty}^{2}r_{u_{\infty}}
+
c_{0}\sum_{i}\alpha_{i}^{2}
+
o(1) \\
= \; &
(\frac{k_{\infty}}{r_{\infty}})^{\frac{n-2}{2}}
(r_{u_{\infty}}+\sum_{i}
\frac{c_{0}(4n(n-1))^{\frac{n-2}{2}}}{K^{\frac{n-2}{2}}(a_{i})})
+o(1)
.
\end{split}
\end{equation}
On the other hand  $r_{u_{\infty}}<0$ by lower semicontinuity, whence for suitable $\alpha>0$
$$
\alpha u_{\infty} \in X
\; \text{ and } \;
0<\inf_{X}J\leq J(\alpha u_{\infty})
=
\frac{-k_{u_{\infty}}}{(-r_{u_{\infty}})^{\frac{n}{n-2}}}
=
(-r_{u_{\infty}})^{-\frac{2}{n}}.
$$
Thus
$r_{u_{\infty}}>- (\inf_{X}J)^{-\frac{n}{2}}$ in contradiction to
\eqref{r_on_X_no_blow_up} for $0<\sup_M K \ll 1$.
\end{proof}
\begin{remark}
The required smallness of $\sup_M K>0$ in Lemma \ref{lem_no_blow_up_for_J}
is determined by $\inf_{X}J$. On the other hand,
since we assume the validity of an A-B-inequality,
this infimum is lower bounded as follows. From
$$
1=\Vert u \Vert_{L^{\frac{2n}{n-2}}}
\leq Ar + B\vert k \vert^{\frac{n-2}{n}}
\leq B\vert k \vert^{\frac{n-2}{n}}
$$
we find
$$J=\frac{-k}{(-r)^{\frac{n}{n-2}}}\geq (-Br)^{-\frac{n}{n-2}},\ \text{while}\ r\geq -\Vert u \Vert_{L^{2}}^{2}\geq -c\Vert u \Vert_{L^{\frac{2n}{n-2}}}^{\frac{n-2}{n}}=-c$$
and so
$\inf_{X}J\geq \gamma B^{-\frac{n}{n-2}}$, $\gamma=\gamma(M)$.
\end{remark}

Similarly from Proposition \ref{prop_blow_up_analysis}
we have the following rough energy estimate.
\begin{proposition}\label{prop_rough_energy}
Suppose, that an A-B-inequality holds on $H^1(M)$.
Then for $0<\sup_M K$  sufficiently small, if a flow line for $I$ along a sequence in time
blows up as
$u=\alpha u_{\infty}+\alpha^{i}\varphi_{i}+v$,
there holds
$$
I(u)\longrightarrow
(
E(u_{\infty})
+
c_{0}\sum (\frac{4n(n-1)}{K(a_{i_{\infty}})})^{\frac{n-2}{2}}
)^{\frac{2}{n}},
$$
where
\begin{enumerate}[label=(\roman*)]
\item $E(u_{\infty})=-J^{-\frac{n-2}{2}}(u_{\infty}/\Vert u_{\infty} \Vert_{L^{\frac{2n}{n-2}}})$,
if $u_{\infty}/\Vert u_{\infty} \Vert_{L^{\frac{2n}{n-2}}}\in X$
\item $E(u_{\infty})=I^{\frac{n}{2}}(u_{\infty}/\Vert u_{\infty} \Vert_{L^{\frac{2n}{n-2}}})$,
if $u_{\infty}/\Vert u_{\infty} \Vert_{L^{\frac{2n}{n-2}}}\in Y$.
\end{enumerate}

\end{proposition}	
\begin{proof}
From Proposition \ref{prop_blow_up_analysis} we find
\begin{equation*}
\begin{split}
I(u)
= \; &
\frac{\alpha^{2}\int L_{g_{0}}u_{\infty}u_{\infty}d\mu_{g_{0}} + c_{0}\sum \alpha_{i}^{2} }
{
(
\alpha^{\frac{2n}{n-2}}\int Ku_{\infty}^{\frac{2n}{n-2}}d\mu_{g_{0}}
+
c_{1}K(a_{i})\alpha_{i}^{\frac{2n}{n-2}})^{\frac{n-2}{n}}
}
+
o(1).
\end{split}
\end{equation*}
where
\begin{equation}\label{constants_c_1_c_0}
c_{1}=\int_{\R^{n}}\frac{dx}{(1+r^{2})^{n}}
\; \text{ and } \;
c_{0}=4n(n-1)c_{1},
\end{equation}
cf. Lemmata \ref{lem_L_g_0_of_bubble}, \ref{lem_interactions}. As
$L_{g_{0}}u_{\infty}=Ku_{\infty}^{\frac{n+2}{n-2}}$, see Proposition \ref{prop_blow_up_analysis},
we may write
\begin{equation*}
\begin{split}
I(u)
= \; &
\frac{\alpha^{2}k_{u_{\infty}} + c_{0}\sum \alpha_{i}^{2} }
{
(
\alpha^{\frac{2n}{n-2}}k_{u_{\infty}}
+
c_{1} \sum K(a_{i})\alpha_{i}^{\frac{2n}{n-2}})^{\frac{n-2}{n}}
}
+
o(1)
\end{split}
\end{equation*}
and obtain, again from Proposition \ref{prop_blow_up_analysis},
\begin{equation*}
\begin{split}
I(u)
= \; &
\frac{
(\frac{k_{\infty}}{r_{\infty}})^{\frac{n-2}{2}}k_{u_{\infty}}
+
c_{0}\sum (\frac{4n(n-1)k_{\infty}}{r_{\infty}K(a_{i_{\infty}})})^{\frac{n-2}{2}}
}
{
(
(\frac{k_{\infty}}{r_{\infty}})^{\frac{n}{2}}k_{u_{\infty}}
+
c_{1}\sum K(a_{i_{\infty}})(\frac{4n(n-1)k_{\infty}}{r_{\infty}K(a_{i_{\infty}})})^{\frac{n}{2}})^{\frac{n-2}{n}}
}
+
o(1) \\
= \; &
\frac{
k_{u_{\infty}}
+
c_{0}\sum (\frac{4n(n-1)}{K(a_{i_{\infty}})})^{\frac{n-2}{2}}
}
{
(
k_{u_{\infty}}
+
c_{1}\sum K(a_{i_{\infty}})(\frac{4n(n-1)}{K(a_{i_{\infty}})})^{\frac{n}{2}})^{\frac{n-2}{n}}
}
+
o(1),
\end{split}
\end{equation*}
which due to $c_{0}=4n(n-1)c_{1}$ simplifies to
$$
I(u)
=
(
k_{u_{\infty}}
+
c_{0}\sum (\frac{4n(n-1)}{K(a_{i_{\infty}})})^{\frac{n-2}{2}}
)^{\frac{2}{n}}
+
o(1).
$$
Note, that
$r_{u_{\infty}}=k_{u_{\infty}}$ and thus
$u_{\infty}/\Vert u_{\infty} \Vert_{L^{\frac{2n}{n-2}}}$ is either in $X$ or $Y$, whence
\begin{enumerate}[label=(\roman*)]
\item $J(u_{\infty}/\Vert u_{\infty} \Vert_{L^{\frac{2n}{n-2}}})=(-k_{u_{\infty}})^{-\frac{2}{n-2}}$,
if 	$k_{u_{\infty}}<0$, i.e. $u_{\infty}/\Vert u_{\infty} \Vert_{L^{\frac{2n}{n-2}}} \in X$
\item
$I(u_{\infty}/\Vert u_{\infty} \Vert_{L^{\frac{2n}{n-2}}})=k_{u_{\infty}}^{\frac{2}{n}}$,
if $k_{u_{\infty}} >0$, i.e. $u_{\infty}/\Vert u_{\infty} \Vert_{L^{\frac{2n}{n-2}}} \in Y$.
\end{enumerate}
The proof is complete.
\end{proof}
Note,
that from Proposition \ref{prop_rough_energy} evidently the least possible blow-up energy $I_{\infty}$ for $I$
occurs in the single bubbling case
$u=\alpha_{\infty}u_{\infty}+\alpha \varphi_{a,\lambda}+v$,
and, when $K(a)=\max K>0$.
On the other hand,
in order to show the existence of a second solution besides a global minimizer
$$u_{0}\in X
\; \text{ with } \;
\inf_{X}J=\min_{X}J=J(u_{0}),
$$
which by Theorem \ref{thm_minimize_J} exists, we may argue by contradiction and assume,
that $u_{0}$ is the only
solution to
$$L_{g_{0}}u=\frac{1}{\beta} Ku^{\frac{n+2}{n-2}}, \beta>0.$$
Then in return by
Proposition \ref{prop_blow_up_analysis} the only possible simple
blow-up scenario for $I$ is a mixed bubbling of type
$
u=\alpha u_{0}+\alpha_{1}\varphi_{a_{1},\lambda_{1}}+v$,
where
$u_{\infty}=u_{0}$.
\begin{proposition}\label{prop_below_I_infty}
Under the assumptions of Theorem \ref{thm_second_solution} for suitable
$$a\in \{ K=\max K \},$$
choices $\alpha_{0},\alpha_{1}>0$ and
$u=\alpha_{0}u_{0}+\alpha_{1}\varphi_{a,\lambda}$
there holds
$$
I(u/\Vert u \Vert_{L^{\frac{2n}{n-2}}})< I_{\infty}
,
$$
where $I_{\infty}$ denotes the least possible blow-up energy of $I$.
\end{proposition}	
In particular Theorem \ref{thm_second_solution} follows,
since by Proposition \ref{prop_below_I_infty}
we may consider a minimizing sequence
$(u_{k})\subset\{ I<I_{\infty} \}$,
which by virtue of Proposition \ref{prop_blow_up_analysis} then leads to a minimizing sequence
$(w_{k})\subset\{ I<I_{\infty} \}$ of solutions $\partial I(w_{k})=0$,
which again by Proposition \ref{prop_blow_up_analysis} converges to a minimizer of $I$.
\begin{proof}[Proof of Proposition \ref{prop_below_I_infty}]
We consider
$$u=\gamma(\alpha_{0} u_{0}+\alpha_{1}\varphi_{a_{1},\lambda_{1}})\in Y$$
for a suitable choice of $\gamma>0$ and have
\begin{equation*}
\begin{split}
I(u)
= \; &
\frac{
\int L_{g_{0}}(\alpha_{0}u_{0}+\alpha_{1}\varphi_{a_{1},\lambda_{1}})
(\alpha_{0}u_{0}+\alpha_{1}\varphi_{a_{1},\lambda_{1}})d\mu_{g_{0}}
}
{(\int K(\alpha_{0}u_{0}+\alpha_{1}\varphi_{a_{1},\lambda_{1}})^{\frac{2n}{n-2}}d\mu_{g_{0}})^{\frac{n-2}{n}}}
=
\frac{N}{D}.
\end{split}
\end{equation*}
We first consider the case
$3\leq n \leq 5$.
From Lemma \ref{lem_L_g_0_of_bubble} we then find
\begin{equation*}
\begin{split}
N
= \; &
\alpha_{0}^{2}r_{u_{0}}
+
2\alpha_{0}\alpha_{1} \int L_{g_{0}}u_{0}\varphi_{a_{1},\lambda_{1}}
+
\alpha_{1}^{2}\int L_{g_{0}}\varphi_{a_{1},\lambda_{1}}\varphi_{a_{1},\lambda_{1}}
d\mu_{g_{0}} \\
= \; &
\alpha_{0}^{2}r_{u_{0}}
+
2\alpha_{0}\alpha_{1} \int L_{g_{0}}u_{0}\varphi_{a_{1},\lambda_{1}}
+
4n(n-1)\alpha_{1}^{2}\int \varphi_{a_{1},\lambda_{1}}^{\frac{2n}{n-2}}d\mu_{g_{0}}
+
o(\frac{1}{\lambda_{1}^{\frac{n-2}{2}}})
\end{split}
\end{equation*}
and thus from \eqref{Expansion_Green_Function}-\eqref{Bubble_Definition}  by expansion
\begin{equation*}
\begin{split}
N
= \; &
\alpha_{0}^{2}r_{u_{0}}
+
2\alpha_{0}\alpha_{1} \int L_{g_{0}}u_{0}\varphi_{a_{1},\lambda_{1}}
+
c_{0}\alpha_{1}^{2}
+
o(\frac{1}{\lambda_{1}^{\frac{n-2}{2}}}),
\end{split}
\end{equation*}
where $c_{0}=4n(n-1)c_{1}$, cf. \eqref{constants_c_1_c_0}.
On the other hand
$$
L_{g_{0}}u_{0}=\frac{1}{\beta}Ku_{0}^{\frac{n+2}{n-2}}
\; \text{ for some } \;
\beta>0
$$
and again  we derive by expansion
up to some $o(\frac{1}{\lambda_{1}^{\frac{n-2}{2}}})$
\begin{equation}\label{estimate_on_D}
\begin{split}
D^{\frac{n}{n-2}}
= \; &
\int K
(
\alpha_{0}^{\frac{2n}{n-2}}u_{0}^{\frac{2n}{n-2}}
+
\alpha_{1}^{\frac{2n}{n-2}}\varphi_{a_{1},\lambda_{1}}^{\frac{2n}{n-2}} \\
& \quad\quad\;\;+
\frac{2n\alpha_{1}\alpha_{0}^{\frac{n+2}{n-2}}}{n-2}u_{0}^{\frac{n+2}{n-2}}\varphi_{a_{1},\lambda_{1}}
+
\frac{2n\alpha_{0}\alpha_{1}^{\frac{n+2}{n-2}}}{n-2}\varphi_{a_{1},\lambda_{1}}^{\frac{n+2}{n-2}}u_{0}
)
d\mu_{g_{0}} \\
= \; &
\beta\alpha_{0}^{\frac{2n}{n-2}}r_{u_{0}}
+
c_{1}\alpha_{1}^{\frac{2n}{n-2}}(K(a_{1})+O(\frac{1}{\lambda_{1}^{2}}))
+
\frac{2nc_{4}}{n-2}\alpha_{0}\alpha_{1}^{\frac{n+2}{n-2}}\frac{u_{0}(a_{1})}{\lambda_{1}^{\frac{n-2}{2}}}
\\
& \; +
\frac{2n}{n-2}\beta\alpha_{1}\alpha_{0}^{\frac{n+2}{n-2}}\int L_{g_{0}}u_{0}\varphi_{a_{1},\lambda_{1}}d\mu_{g_{0}}
,
\end{split}
\end{equation}
where
$c_{4}=\int_{\R^{n}}\frac{dx}{(1+r^{2})^{\frac{n+2}{2}}}$.
Moreover
\begin{equation*}%\label{estimate_low_dimension}
\frac{1}{\lambda_{1}^{2}}=o(\frac{1}{\lambda_{1}^{\frac{n-2}{2}}}),
\; \text{ as }
\; 3\leq n \leq 5,
\end{equation*}
and, since
$\int L_{g_{0}}u_{0}\varphi_{a_{1},\lambda_{1}}d\mu_{g_{0}}=O(\frac{1}{\lambda_{1}^{\frac{n-2}{2}}})$,
we obtain up to some $o(\frac{1}{\lambda_{1}^{\frac{n-2}{2}}})$
\begin{equation*}%\label{I_expansion}
\begin{split}
I(u)
= \; &
\frac{
\alpha_{0}^{2}r_{u_{0}}
+
c_{0}\alpha_{1}^{2}
}
{
(\beta\alpha_{0}^{\frac{2n}{n-2}}r_{u_{0}}
+
c_{1}K(a_{1})\alpha_{1}^{\frac{2n}{n-2}})^{\frac{n-2}{n}}
}
\\
& -
\frac{
2c_{4}
(
\alpha_{0}^{2}r_{u_{0}}
+
c_{0}\alpha_{1}^{2}
)
\alpha_{0}\alpha_{1}^{\frac{n+2}{n-2}}
}
{
(\beta\alpha_{0}^{\frac{2n}{n-2}}r_{u_{0}}
+
c_{1}K(a_{1})\alpha_{1}^{\frac{2n}{n-2}})^{\frac{n-2}{n}+1}
}
\frac{u_{0}(a_{1})}{\lambda_{1}^{\frac{n-2}{2}}}
\\
& +
\frac
{
2\alpha_{0}\alpha_{1}\int L_{g_{0}}u_{0}\varphi_{a_{1},\lambda_{1}}d\mu_{g_{0}}
}
{(\beta\alpha_{0}^{\frac{2n}{n-2}}r_{u_{0}}
+
c_{1}K(a_{1})\alpha_{1}^{\frac{2n}{n-2}})^{\frac{n-2}{n}}
}
[1-\frac{
(
\alpha_{0}^{2}r_{u_{0}}
+
c_{0}\alpha_{1}^{2}
)
\alpha_{0}^{\frac{4}{n-2}}\beta
}
{
\beta\alpha_{0}^{\frac{2n}{n-2}}r_{u_{0}}
+
c_{1}K(a_{1})\alpha_{1}^{\frac{2n}{n-2}}
}
].
\end{split}
\end{equation*}
Choosing e.g.
$$
\alpha_{0}^{\frac{4}{n-2}}=1/\beta
\; \text{ and } \;
\alpha_{1}^{\frac{4}{n-2}}=\frac{c_{0}}{c_{1}K(a_{1})}=\frac{4n(n-1)}{K(a_{1})},
$$
the latter summand vanishes and, as $u_{0}>0$, we obtain
\begin{equation*}
\begin{split}
I(u)
= \; &
(\beta^{-\frac{n-2}{2}}	r_{u_{0}}
+
c_{0}(\frac{4n(n-1)}{K(a_{i})})^{\frac{n-2}{2}}
)^{\frac{2}{n}}
-
O^{+}(\frac{1}{\lambda_{1}^{\frac{n-2}{2}}}).
\end{split}
\end{equation*}
Recalling  $L_{g_{0}}u_{0}=\frac{1}{\beta}Ku_{0}^{\frac{n+2}{n-2}}$ and thus
\begin{equation*}
\begin{split}
J(\frac{u_{0}}{\Vert u_{0} \Vert_{L^{\frac{2n}{n-2}}}})
= \; &
\frac{-k_{u_{0}}}{(-r_{u_{0}})^{\frac{n}{n-2}}}
=
\frac{-\beta r_{u_{0}}}{(-r_{u_{0}})^{\frac{n}{n-2}}}
=
\frac{\beta}{(-r_{u_{0}})^{\frac{2}{n-2}}}
=
(-\beta^{-\frac{n-2}{2}}r_{u_{0}})^{-\frac{2}{n-2}},
\end{split}
\end{equation*}
we finally derive
\begin{equation*}
\begin{split}
I(u)
= \; &
(-J^{-\frac{n-2}{2}}(u_{0})+c_{0}(\frac{4n(n-1)}{K(a_{1})}))^{\frac{2}{n}}
-
O^{+}(\frac{1}{\lambda_{1}^{\frac{n-2}{2}}}).
\end{split}
\end{equation*}
Recalling Proposition \ref{prop_rough_energy} (i), the assertion follows for $\lambda \gg 1$
in case $3\leq n\leq 5$. The case $n\geq 6$ with local flatness around $a\in \{ K=\max K \} $
and
$$
\nabla^{l}K(a)=0 \; \text{ for all } \; \frac{n-2}{2} \geq l \in \N
$$
then follows as when $3\leq n \leq 5$ upon replacing in \eqref{estimate_on_D}
$$O(\frac{1}{\lambda^{2}}) \; \text{ by } \; O(\frac{1}{\lambda^{l+1}}),$$
noticing $\frac{1}{\lambda^{l+1}}=o_{\frac{1}{\lambda}}(\frac{1}{\lambda^{\frac{n-2}{2}}})$
and recalling $H_{s,a}\equiv 0$, when expanding.
\end{proof}

\begin{remark}\label{rem_critcal_points_at_infinity}
The argument of the proof reflects the question of existence or non existence
of mixed type critical points at infinity \cite{Bahri_High_Dimensions},
depending on the usual non degeneracy and flatness conditions.
Although here pure type critical points at infinity do not exist,
mixed type ones will generally occur,
provided of course, that classical solutions exist in the first place.
\end{remark}	

\section{Appendix}
First note, that we may write any
$
u
=
\tilde{\alpha} u_{\infty}
+
\tilde{\alpha}^{i}\varphi_{a_{i},\lambda_{i}}
+
\tilde{v}
\in
V(u_{\infty},q,\varepsilon_2)
$
as
\begin{equation*}
u=\alpha u_{\infty}+\alpha^{i}\varphi_{ a_i,\lambda_i}+v
\; \text{ with } \;
v \perp_{L_{g_{0}}}span\{u_{\infty},\varphi_{ a_i,\lambda_i}\}
\end{equation*}
by solving in the $(\alpha,\alpha_{i})$-variables the linear system
\begin{equation*}\label{alpha_alpha_i_system}
\left\{\begin{matrix*}[c]
\quad
\langle u, u_\infty\rangle_{ L_{g_0}} \;\;\;\;\,
& = &
\alpha\langle u_\infty, u_\infty\rangle_{L_{g_0}} \;\;\;\;\,
& + &
\alpha^i\langle \varphi_{a_i, \lambda_i}, u_\infty\rangle_{L_{g_0}} \;\;\;\;\,
\\
\quad
\langle u, \varphi_{a_1, \lambda_1}\rangle_{L_{g_0}}
& = &
\alpha\langle u_\infty, \varphi_{a_1,\lambda_1}\rangle_{L_{g_0}}
& + &
\alpha^i\langle \varphi_{a_i, \lambda_i}, \varphi_{a_1, \lambda_1}\rangle_{L_{g_0}}
\\
\quad
\vdots &&\vdots&& \vdots
\\
\quad
\langle u, \varphi_{a_q, \lambda_q}\rangle_{L_{g_0}}
& = &
\alpha\langle u_\infty, \varphi_{a_q, \lambda_q}\rangle_{L_{g_0}}
& + &
\alpha^i\langle \varphi_{a_i, \lambda_i}, \varphi_{a_q, \lambda_q}\rangle_{L_{g_0}}
\end{matrix*}
\right.	
\end{equation*}
which due to
$
\vert \langle u_\infty, u_\infty\rangle_{L_{g_0}} \vert, \langle \varphi_{a_i,\lambda_i}, \varphi_{a_i, \lambda_i}\rangle_{L_{g_0}}
\geq
c>0
$
and
\begin{equation*}%\label{alpha_alpha_i_system_off_diagonal_entries}
\langle \varphi_{a_i, \lambda_i}, u_\infty\rangle_{L_{g_0}}, \langle L_{g_0}\varphi_{a_i, \lambda_i}, \varphi_{a_j, \lambda_j}\rangle_{L_{g_0}}
=
o_{\varepsilon_2}(1)
\; \text{ for } \;
j\neq i
\end{equation*}
clearly admits a unique solution $(\alpha, \alpha_i)$,
whose dependence on $(a_{i},\lambda_{i})$ we clarify as follows.
From the above let us write some \textit{fixed} $u\in V(u_{\infty},q,\varepsilon_2)$ as
$$
u=\alpha u_{\infty}+\alpha^{i}\varphi_{a_i, \lambda_i}+v,\; v \perp_{L_{g_{0}}} span\{ u_{\infty},\varphi_{a_i, \lambda_i} \}.
$$
Varying in $(a_{i},\lambda_{i})$ this representation of $u$, due to $\langle v,u_{\infty} \rangle_{L_{g_{0}}}=0$ we have
\begin{equation}\label{alpha_derivative_tested}
\begin{split}
0
= \; &
\lambda_{i}\partial_{\lambda_{i}}
\langle u,u_{\infty} \rangle_{L_{g_{0}}}
=
\lambda_{i}\partial_{\lambda_{i}}
\langle \alpha u_{\infty}+\alpha^{j}\varphi_{a_j, \lambda_j}	,u_{\infty} \rangle_{L_{g_{0}}} \\
= \; &
\lambda_{i}\partial_{\lambda_{i}}\alpha \langle u_{\infty},u_{\infty} \rangle_{L_{g_{0}}} \\
&  +
\lambda_{i}\partial_{\lambda_{i}}\alpha^{j}
\langle \varphi_{a_{j},\lambda_{j}},u_{\infty}\rangle_{L_{g_{0}}}
+
\alpha_{i}\langle \lambda_{i}\partial_{\lambda_{i}}\varphi_{a_{i},\lambda_{i}},u_{\infty}
\rangle_{L_{g_{0}}},
\end{split}
\end{equation}
whence
\begin{equation}\label{alpha_derivative_estimated}
\lambda_{i}\partial_{\lambda_{i}}\alpha \langle u_{\infty},u_{\infty} \rangle_{L_{g_{0}}}
=
o_{\varepsilon_2}(\sum_{i,j}\vert \lambda_{i}\partial_{\lambda_{i}}\alpha_{j} \vert )
+
O(\frac{1}{\lambda_{i}^{\frac{n-2}{2}}}).
\end{equation}
Likewise for $i\neq j$ and due to
$\langle v,\varphi_{a_{j},\lambda_{j}}\rangle_{L_{g_{0}}}=0$
\begin{equation}\label{alpha_i_mixed_derivative_tested}
\begin{split}
0
= \; &
\lambda_{i}\partial_{\lambda_{i}}
\langle u,\varphi_{a_{j},\lambda_{j}}\rangle_{L_{g_{0}}}
=
\lambda_{i}\partial_{\lambda_{i}}
\langle \alpha u_{\infty}+\alpha^{k}\varphi_{a_k, \lambda_k},
\varphi_{a_{j},\lambda_{j}} \rangle_{L_{g_{0}}} \\
= \; &
\lambda_{i}\partial_{\lambda_{i}}\alpha_{j}
\langle \varphi_{a_{j},\lambda_{j}},\varphi_{a_{j},\lambda_{j}}\rangle_{L_{g_{0}}}
+
\alpha_{i}\langle \lambda_{i}\partial_{\lambda_{i}}\varphi_{a_{i},\lambda_{i}},\varphi_{a_{j},\lambda_{j}} \rangle_{L_{g_{0}}}
 \\
& +
\lambda_{i}\partial_{\lambda_{i}}\alpha
\langle u_{\infty},\varphi_{a_{j},\lambda_{j}} \rangle_{L_{g_{0}}}
+
\sum_{j\neq k=1}^{q} \lambda_{i}\partial_{\lambda_{i}}\alpha _{k}
\langle \varphi_{a_{k},\lambda_{k}},\varphi_{a_{j},\lambda_{j}} \rangle_{L_{g_{0}}},
\end{split}
\end{equation}
whence
\begin{equation}\label{alpha_i_mixed_derivative_estimated}
\begin{split}
\lambda_{i}\partial_{\lambda_{i}}\alpha_{j}
\langle \varphi_{a_{j},\lambda_{j}},\varphi_{a_{j},\lambda_{j}}\rangle_{L_{g_{0}}}
=
o_{\varepsilon_2}
(\vert \lambda_{i}\partial_{\lambda_{i}}\alpha  \vert
+
\sum_{j\neq k=1}^{q}\vert \lambda_{i}\partial_{\lambda_{i}}\alpha_{k} \vert
)
+
O(\varepsilon_{i,j})
,
\end{split}
\end{equation}
cf. Lemma \ref{lem_interactions}.
Similarly we compute on one hand
\begin{equation*}
\begin{split}
\lambda_{i}\partial_{\lambda_{i}}\langle u,\varphi_{a_i, \lambda_i} \rangle_{L_{g_{0}}}
= \; &
\langle u,\lambda_{i}\partial_{\lambda_{i}} \varphi_{a_i, \lambda_i}\rangle_{L_{g_{0}}}\\
=\; &
\langle \alpha u_{\infty}+\alpha^{j}\varphi_{a_j, \lambda_j}+v,
\lambda_{i}\partial_{\lambda_{i}} \varphi_{a_i, \lambda_i}\rangle_{L_{g_{0}}},
\end{split}
\end{equation*}
while on the other due $\langle v , \varphi_{a_i, \lambda_i} \rangle_{L_{g_{0}}}=0$
\begin{equation*}
\begin{split}
\lambda_{i}\partial_{\lambda_{i}} & \langle u  ,  \varphi_{a_i, \lambda_i} \rangle_{L_{g_{0}}}
= \;
\lambda_{i}\partial_{\lambda_{i}}
\langle
\alpha u_{\infty}+\alpha^{j}\varphi_{a_j, \lambda_j},
\varphi_{a_i, \lambda_i} \rangle_{L_{g_{0}}}  \\
= \; &
\langle \alpha u_{\infty} + \alpha^{j}\varphi_{a_{j},\lambda_{j}},\lambda_{i}\partial_{\lambda_{i}}\varphi_{a_{i},\lambda_{i}}\rangle_{L_{g_{0}}} \\
& +
\lambda_{i}\partial_{\lambda_{i}}\alpha_{i} \langle \varphi_{a_i, \lambda_i},\varphi_{a_i, \lambda_i}\rangle_{L_{g_{0}}}
+
\alpha_{i}\langle \lambda_{i}\partial_{\lambda_{i}}\varphi_{a_{i},\lambda_{i}},\varphi_{a_{i},\lambda_{i}}\rangle_{L_{g_{0}}}
\\
& +
\lambda_{i}\partial_{\lambda_{i}}\alpha
\langle u_{\infty},\varphi_{a_{i},\lambda_{i}}\rangle_{L_{g_{0}}}
+
\sum_{i\neq j=1}^{q}\lambda_{i}\partial_{\lambda_{i}}\alpha_{j}
\langle \varphi_{a_{j},\lambda_{j}},\varphi_{a_{i},\lambda_{i}}
\rangle_{L_{g_{0}}},
\end{split}
\end{equation*}
and so we conclude by subtracting
\begin{equation}\label{alpha_i_derivative_tested}
\begin{split}
\lambda_{i}\partial_{\lambda_{i}}\alpha_{i}
&
\langle \varphi_{a_i, \lambda_i}  ,  \varphi_{a_i, \lambda_i}\rangle_{L_{g_{0}}} \\
= \; &-
\lambda_{i}\partial_{\lambda_{i}}\alpha
\langle u_{\infty},\varphi_{a_{i},\lambda_{i}}\rangle_{L_{g_{0}}}
-
\sum_{i\neq j=1}^{q}\lambda_{i}\partial_{\lambda_{i}}\alpha_{j}
\langle \varphi_{a_{j},\lambda_{j}},\varphi_{a_{i},\lambda_{i}}
\rangle_{L_{g_{0}}}\\
& \;
-
\alpha_{i}\langle \varphi_{a_{i},\lambda_{i}},\lambda_{i}\partial_{\lambda_{i}}\varphi_{a_{i},\lambda_{i}}\rangle_{L_{g_{0}}}
+
\langle v,\lambda_{i}\partial_{\lambda_{i}}\varphi_{a_{i},\lambda_{i}}
\rangle_{L_{g_{0}}},
\end{split}	
\end{equation}
which we estimate as
\begin{equation}\label{alpha_i_derivative_estimated}
\begin{split}
\lambda_{i}\partial_{\lambda_{i}}\alpha_{i}
\langle \varphi_{a_i, \lambda_i} , & \varphi_{a_i, \lambda_i}\rangle_{L_{g_{0}}}
=
o_{\varepsilon_2}(\vert \lambda_{i}\partial_{\lambda_{i}}\alpha \vert
+
\sum_{i\neq j =1}^{q}\vert \lambda_{i}\partial_{\lambda_{i}}\alpha_{j} \vert )\\
& +
\langle v,\lambda_{i}\partial_{\lambda_{i}}\varphi_{a_{i},\lambda_{i}}
\rangle_{L_{g_{0}}}
-
\alpha_{i}\langle \varphi_{a_{i},\lambda_{i}},\lambda_{i}\partial_{\lambda_{i}}\varphi_{a_{i},\lambda_{i}}\rangle_{L_{g_{0}}}.
\end{split}
\end{equation}
Since
$$
\langle u_{\infty},u_{\infty}\rangle_{L_{g_{0}}}=r_{u_{\infty}}\neq 0
\; \text{ and } \;
\langle \varphi_{a_i, \lambda_i},\varphi_{a_i, \lambda_i}\rangle_{L_{g_{0}}} \neq o(1),
$$
we then find from
\eqref{alpha_derivative_estimated},
\eqref{alpha_i_mixed_derivative_estimated}
and \eqref{alpha_i_derivative_estimated}
\begin{equation}\label{alpha_and_alpha_i_derivatives_estimate}
\begin{split}
\sum_{i}
(
\vert \lambda_{i}\partial_{\lambda_{i}}\alpha \vert
+
\sum_{j}\vert \lambda_{i}\partial_{\lambda_{i}}\alpha_{j} \vert
)
= \; &
O(\sum_{i}\frac{1}{\lambda_{i}^{\frac{n-2}{2}}}+\sum_{i\neq j}\varepsilon_{i,j} + \Vert v \Vert) \\
& +
O(\sum_{i}
\vert
\langle \varphi_{a_{i},\lambda_{i}},\lambda_{i}\partial_{\lambda_{i}}\varphi_{a_{i},\lambda_{i}}\rangle_{L_{g_{0}}} \vert
) \\
= \; &
o_{\varepsilon_2}(1)
\end{split}
\end{equation}
and therefore, if $Cond_{n}$ holds at all $a_{i}, \, i=1,\ldots,q$,
then
\begin{equation}\label{alpha_and_alpha_i_derivatives_estimate_with_flatness}
\begin{split}
\sum_{i}
(
\vert \lambda_{i}\partial_{\lambda_{i}}\alpha \vert
+
\sum_{j}\vert \lambda_{i}\partial_{\lambda_{i}}\alpha_{j} \vert
)
= \; &
O(\sum_{i}\frac{1}{\lambda_{i}^{\frac{n-2}{2}}}+\sum_{i\neq j}\varepsilon_{i,j} + \Vert v \Vert),
\end{split}
\end{equation}
as follows from Lemma \ref{lem_L_g_0_of_bubble}, see Lemma \ref{lem_interactions} (iv).
Finally and by the same reasoning \eqref{alpha_and_alpha_i_derivatives_estimate} and \eqref{alpha_and_alpha_i_derivatives_estimate_with_flatness}
hold for $\lambda_{i}\partial_{\lambda_{i}}$ replaced by $\sfrac{\nabla_{a_{i}}}{\lambda_{i}}$.

 \

With this in mind the proof of Lemma \ref{lem_optimal_choice} decomposes into three steps, first showing that
the infimum is attained in the interior, secondly that the resulting minimum is unique and finally justifying the
estimates yielding  sufficient \textit{almost}-orthogonalities for the $v$-part.

\begin{proof}[\textbf{Proof of Lemma \ref{lem_optimal_choice}}]
Consider some fixed
$$
u=\hat{\alpha}u_{\infty}
+
\hat{\alpha}^{i}\varphi_{\hat{a}_{i},\hat{\lambda}_{i}}+\hat{v}
\in V(u_\infty, q,\varepsilon_{2}).
$$
Then at some $(\alpha,\alpha_{i},\lambda_{i},a_{i})\in A_{u}(u_{\infty},q,2 \varepsilon_{1})$
\begin{equation*}
\begin{split}
\inf_
{
(
\tilde a_{i}, \tilde\lambda_{i})
\in \Pi_{(a_i, \lambda_i)}A_{u}(u_{\infty},q,2\varepsilon_{1})
}
\int u^{\frac{4}{n-2}}\vert u- \tilde{\alpha }u_{\infty}-\tilde{\alpha}^{i}\varphi_{\tilde{a}_{i},\tilde{\lambda}_{i}} \vert^{2} d\mu_{g_{0}}
\end{split}
\end{equation*}
is attained and, since $\Vert \hat{v} \Vert\leq \varepsilon_{2}$, there holds
\begin{equation*}
\begin{split}
o_{\varepsilon_{2}}(1)
=  &
\int u^{\frac{4}{n-2}}\vert u- \alpha  u_{\infty}-\alpha^{i}\varphi_{a_{i},\lambda_{i}} \vert^{2}d\mu_{g_{0}}
\\
= \; &
o_{\varepsilon_{2}}(1)
+
\int
\vert \hat{\alpha}u_{\infty}
+
\hat{\alpha}^{i}\varphi_{\hat{a}_{i},\hat{\lambda_{i}}}
\vert^{\frac{4}{n-2}}
\\  & \quad\quad\quad\quad\quad\quad\quad\quad \;
(
\hat{\alpha}u_{\infty} -\alpha u_{\infty}
+
\hat{\alpha}^{i}\varphi_{\hat{a}_{i},\hat{\lambda}_{i}}
-
\alpha^{j}\varphi_{a_{j},\lambda_{j}}
)^2d\mu_{g_{0}}
,
\end{split}	
\end{equation*}
whence after possibly relabelling the indices $j=1,\ldots,q$ for
\begin{equation*}
\begin{split}
&A=\vert \alpha	-\hat{\alpha}\vert,
\, A_i=\vert \alpha_i-\hat{\alpha}_i\vert,
\,
L_i=\vert\frac{\hat{\lambda}_i}{\lambda_i}-1\vert,
\,
D_i=\sqrt{\lambda_i\hat{\lambda}_i d_{g_{0}}^2(a_i, \hat{a}_i)}
\end{split}
\end{equation*}
we necessarily have
\begin{equation}\label{minimizer_difference_estimate}
A+\sum_{i}(A_{i}+L_{i}+D_{i})=o_{\varepsilon_{2}}(1).
\end{equation}
In particular, since
\begin{equation}\label{v_hat_v_difference}
v=u-\alpha u_{\infty}-\alpha^{i}\varphi_{a_{i},\lambda_{i}}
=
\hat{\alpha} u_{\infty} - \alpha u_{\infty}
+
\hat{\alpha}^{i}\varphi_{\hat{a}_{i},\hat{\lambda}_{i}}
-
\alpha^{i}\varphi_{a_{i},\lambda_{i}}
+
\hat{v},
\end{equation}
we find $\Vert v \Vert = o_{\varepsilon_{2}}(1)$ and thus for
$\varepsilon_{2}$ sufficiently small with $o_{\varepsilon_{2}}(1) < \varepsilon_{1}$, that
$$
(\alpha,\alpha_{i},\lambda_{i},a_{i})
\in
A_{u}(u_{\infty},q,o_{\varepsilon_{2}}(1))
\subset
A_{u}(u_{\infty},q,\varepsilon_{1})
\subset
A_{u}(u_{\infty},q,2\varepsilon_{1})
,
$$
i.e. the infimum is attained as an interior minimum.
To show uniqueness as a critical point of an interior minimizer, assume there are two, say
\begin{equation}\label{two_representations_of_u}
u=
\hat{\alpha}u_{\infty}
+
\hat{\alpha}^{i}\varphi_{\hat{a}_{i},\hat{\lambda_{i}}}+\hat{v}
=
\alpha u_{\infty}+\alpha^{i}\varphi_{a_i, \lambda_i}+v
\in
V(u_{\infty},q,\varepsilon_{1}).
\end{equation}
Then \eqref{minimizer_difference_estimate} holds, by construction we have
$$
\hat{v}\perp_{L_{g_{0}}} span\{u_{\infty},\varphi_{\hat{a}_{i},\hat{\lambda}_{i}}\},
\;
v\perp_{L_{g_{0}}} span \{ u_{\infty},\varphi_{a_{i},\lambda_{i}} \}
$$
and due to minimality,
taking the derivatives in $(\hat{\lambda}_{i},\hat{a}_{i})$ or $(\lambda_{i},a_{i})$,
\begin{equation}\label{lambda_hat_lambda_orthogonalities}
\begin{split}
0
= \; &
\int u^{\frac{4}{n-2}}\hat{v}\hat{\lambda}_{i}\partial_{\hat{\lambda}_{i}}(\hat{\alpha} u_{\infty}+\hat{\alpha}^{j}\varphi_{\hat{a}_{j},\hat{\lambda}_{j}})d\mu_{g_{0}} \\
= \; &
\int u^{\frac{4}{n-2}}v \lambda_{i}\partial_{\lambda_{i}}(\alpha u_{\infty}+\alpha^{j}\varphi_{a_{i},\lambda_{i}})d\mu_{g_{0}}
\end{split}	
\end{equation}
and
\begin{equation}\label{a_hat_a_orthogonalities}
\begin{split}
0
= \; &
\int u^{\frac{4}{n-2}}\hat{v}\frac{\nabla_{\hat{a}_{i}}}{\hat{\lambda}_{i}}(\hat{\alpha} u_{\infty}+\hat{\alpha}^{j}\varphi_{\hat{a}_{j},\hat{\lambda}_{j}})d\mu_{g_{0}} \\
= \; &
\int u^{\frac{4}{n-2}}v \frac{\nabla_{a_{i}}}{\lambda_{i}}(\alpha u_{\infty}+\alpha^{j}\varphi_{a_{j},\lambda_{j}})d\mu_{g_{0}}.
\end{split}
\end{equation}
Uniqueness then follows from
\begin{equation}\label{uniqueness_condition}
A+\sum_{i}(A_i+L_i+D_i)=o_{\varepsilon_{1}}(A+\sum_{i}( A_i+L_i+D_i)).
\end{equation}
To show \eqref{uniqueness_condition}, we start with using
$\langle \hat{v},u_{\infty}\rangle_{L_{g_{0}}}
=0=
\langle v,u_{\infty}\rangle_{L_{g_{0}}}=0$
to get
\begin{equation}\label{alpha_difference_estimate}
\begin{split}
0
= \; &
\int L_{g_0}u_\infty \hat{v}d\mu_{g_{0}}
=
\int L_{g_0}u_\infty
(u-\hat{\alpha}u_\infty-\hat{\alpha}^i\varphi_{\hat{a}_i, \hat{\lambda}_i})
d\mu_{g_{0}}  \\
= \; &
\int L_{g_0}u_\infty
(
\alpha u_\infty+\alpha^i\varphi_{a_i,\lambda_i}
-\hat{\alpha}u_\infty-\hat{\alpha}^i\varphi_{\hat{a}_i, \hat{\lambda}_i})
d\mu_{g_{0}}  \\
=
\; &
(\alpha-\hat{\alpha})\int L_{g_0}u_\infty u_\infty d\mu_{g_{0}}
+
o_{\varepsilon_{1}}(\sum_{i}( A_i+L_i+D_i)),
\end{split}
\end{equation}
and likewise from
$
\langle \hat{v},\varphi_{\hat{a}_{i},\hat{\lambda}_{i}}\rangle_{L_{g_{0}}}
=
0
=
\langle v,\varphi_{a_{i},\lambda_{i}}\rangle_{L_{g_{0}}}=0
$
we obtain
\begin{equation}\label{alpha_i_difference_estimate}
\begin{split}
0
= \; &
\int L_{g_0}\varphi_{\hat{a}_j, \hat{\lambda}_j}\hat{v} d\mu_{g_{0}}
=
\int L_{g_0}\varphi_{\hat{a}_j, \hat{\lambda}_j}
(u-\hat{\alpha}u_\infty-\hat{\alpha}^i\varphi_{\hat{a}_i, \hat{\lambda}_i}) d\mu_{g_{0}}
\\
= \; &
\int L_{g_0}\varphi_{\hat{a}_j, \hat{\lambda}_j}
(
\alpha u_\infty +\alpha^i\varphi_{a_i, \lambda_i}+v
-\hat{\alpha}u_\infty-\hat{\alpha}^i\varphi_{\hat{a}_i, \hat{\lambda}_i})d\mu_{g_{0}}
\\
= \; &
(\alpha_j-\hat{\alpha}_j)
\int L_{g_0}\varphi_{\hat{a}_j, \hat{\lambda}j}
\varphi_{\hat{a}_j, \hat{\lambda}_j}d\mu_{g_{0}}
\\
& +
\alpha_j\int L_{g_0}\varphi_{\hat{a}_j, \hat{\lambda}_j}
(\varphi_{a_j, \lambda_j}-\varphi_{\hat{a}_j, \hat{\lambda}_j})d\mu_{g_{0}} \\
& +
o_{\varepsilon_{1}}(A+\sum_{j\neq i=1}^{q} A_i+\sum_{i=1}^q( L_i+D_i))\\
= \; &
(\alpha_j-\hat{\alpha}_j)
\int L_{g_0}\varphi_{\hat{a}_j, \hat{\lambda}_j}\varphi_{\hat{a}_j,\hat{\lambda}_j}
d\mu_{g_{0}}
+
o_{\varepsilon_{1}}(A+\sum_{i}(A_i+ L_i+D_i)),
\end{split}
\end{equation}
where we made use of Lemma \ref{lem_interactions} and
to treat the term
$$\int L_{g_0}\varphi_{\hat{a}_j, \hat{\lambda}_j}(\varphi_{{a}_j, {\lambda}_j}-\varphi_{\hat{a}_j, \hat{\lambda}_j})d\mu_{g_{0}}$$
also
of \eqref{minimizer_difference_estimate} and
\begin{equation}\label{bubble_difference}
\begin{split}
\varphi_{\hat{a}_i, \hat{\lambda}_i}
-
\varphi_{a_i,\lambda_i}
= \; &
\begin{pmatrix} 
\frac{1}{\lambda_{i}}\nabla_{a_{i}} \\
\lambda_{i}\partial_{\lambda_{i}}
\end{pmatrix}
\varphi_{a_{i},\lambda_{i}}
\begin{pmatrix}
\lambda_i(\hat{a}_i-a_i) \\
\frac{\hat{\lambda}_i}{{\lambda}_i}-1
\end{pmatrix}
%& +
%\frac{1}{2}
%\begin{pmatrix}
%\frac{1}{\lambda_{i}^2}\nabla_{a_{i}}^{2} & \nabla_{a_{i}}\partial_{\lambda_{i}}
%\\
%\partial_{\lambda_{i}}\nabla_{a_{i}}   &  \lambda_{i}^2\partial_{\lambda_{i}^2}
%\end{pmatrix}
%\varphi_{a_{i},\lambda_{i}}
%\begin{pmatrix}
%\lambda_i(\bar{a}_i-\hat{a}_i ) \\
%\frac{\bar{\lambda}_i}{\hat{\lambda}_i}-1
%\end{pmatrix}^{2}
+
o_{L_{i}+D_{i}}(L_{i}+D_{i})
\end{split}
\end{equation}
as well as
\begin{equation*}%\label{inv}
\lambda\partial_\lambda\int_{\mathbb{R}^n}\delta_{a,\lambda}^{\frac{2n}{n-2}}
=
\frac1\lambda\nabla_a\int_{\mathbb{R}^n}\delta_{a,\lambda}^{\frac{2n}{n-2}}
=
0
\;\; \text{ for } \;\;
\delta_{a,\lambda}=(\frac{\lambda}{1+\lambda^{2}\vert \cdot-a \vert^{2} })^{\frac{n-2}{2}}.
\end{equation*}
Combining \eqref{alpha_difference_estimate} and \eqref{alpha_i_difference_estimate} we conclude
\begin{equation}\label{consm}
A + \sum_{i}A_{i} =  o_{\varepsilon_{1}}(A+\sum_{i} ( A_i+L_i+D_i)).
\end{equation}
We proceed by employing \eqref{lambda_hat_lambda_orthogonalities}.
First note, that due to \eqref{v_hat_v_difference}
\begin{equation}\label{v_difference_estimate}
\Vert v -\hat{v}\Vert
=
O(A+\sum_{i}( A_i+L_i+D_i)).
\end{equation}
Secondly \eqref{alpha_derivative_tested}, \eqref{alpha_i_mixed_derivative_tested}, \eqref{alpha_i_derivative_tested} and \eqref{alpha_and_alpha_i_derivatives_estimate}
hold for both representations of $u$ in \eqref{two_representations_of_u},
whence by subtracting the corresponding versions of \eqref{alpha_derivative_tested}, \eqref{alpha_i_mixed_derivative_tested} and \eqref{alpha_i_derivative_tested} we find
\begin{equation}\label{alpha_alpha_i_derivatives_difference}
\begin{split}
\sum_{i}
(\vert \lambda_{i}\partial_{\lambda_{i}}\alpha - \hat{\lambda}_{i}\partial_{\hat{\lambda}_{i}}\hat{\alpha} \vert
+
\sum_{j}
\vert
\lambda_{i}\partial_{\lambda_{i}}\alpha_{j}   &   - \hat{\lambda}_{i}\partial_{\hat{\lambda}_{i}}\hat{\alpha}_{j}
\vert
) \\
= \; &
O(A+\sum_{i}( A_i+L_i+D_i)).
\end{split}	
\end{equation}
Thus, when subtracting in \eqref{lambda_hat_lambda_orthogonalities}, the estimates
\eqref{alpha_and_alpha_i_derivatives_estimate}, \eqref{v_difference_estimate} and \eqref{alpha_alpha_i_derivatives_difference} yield
\begin{equation*}
\begin{split}
\int u^{\frac{4}{n-2}}
(
\alpha_{i}\lambda_{i}\partial_{\lambda_{i}}\varphi_{a_{i},\lambda_{i}} v
-
\hat{\alpha}_{i}\hat{\lambda}_{i}\partial_{\hat{\lambda}_{i}}\varphi_{\hat{a}_{i},\hat{\lambda}_{i}} \hat{v}
)d\mu_{g_{0}}
=
o_{\varepsilon_{1}}(A+\sum_{i}( A_i+L_i+D_i)),
\end{split}
\end{equation*}
whence, as $\vert \alpha_{i}-\hat{\alpha}_{i}\vert=A_{i}$ and $\Vert v \Vert,\Vert \hat{v} \Vert=o_{\varepsilon_{1}}(1)$,  and due to
$$
\Vert
\lambda_{i}\partial_{\lambda_{i}}\varphi_{a_{i},\lambda_{i}}
-
\hat{\lambda}_{i}\partial_{\hat{\lambda}_{i}}\varphi_{\hat{a}_{i},\hat{\lambda}_{i}}
\Vert
=
O(L_{i}+D_{i})
$$
we obtain
$$
\int u^{\frac{4}{n-2}}
\lambda_{i}\partial_{\lambda_{i}}\varphi_{a_{i},\lambda_{i}}(v-\hat{v})d\mu_{g_{0}}
=
o_{\varepsilon_{1}}(A+\sum_{i}( A_i+L_i+D_i)).
$$
Recalling \eqref{two_representations_of_u} and estimating as above, we then get from \eqref{bubble_difference}
\begin{equation*}
\begin{split}
o_{\varepsilon_{1}}(A   &   +\sum_{i}( A_i+L_i+D_i))
=
\int u^{\frac{4}{n-2}}
\lambda_{i}\partial_{\lambda_{i}}\varphi_{a_{i},\lambda_{i}}
(
\varphi_{\hat{a}_{i},\hat{\lambda}_{i}}
-
\varphi_{a_{i},\lambda_{i}}
) d\mu_{g_{0}}
\\
= \; &
(\frac{\hat{\lambda}_{i}}{\lambda_{i}}-1)
\int u^{\frac{4}{n-2}}
\vert \lambda_{i}\partial_{\lambda_{i}}\varphi_{a_{i},\lambda_{i}}\vert^{2}d\mu_{g_{0}}
+
o_{\varepsilon_{1}}(A+\sum_{i}( A_i+L_i+D_i)).
\end{split}
\end{equation*}
By by simple expansions of $u^{\frac{4}{n-2}}$ for $u=\alpha u_{\infty} + \alpha^{i}\varphi_{a_{i},\lambda_{i}}+v$ we thus obtain
\begin{equation*}
L_i=o_{\varepsilon_{1}}(A+\sum_{i} (A_i+L_i+D_i))
\end{equation*}
and by analogous arguments, employing \eqref{a_hat_a_orthogonalities} instead of \eqref{lambda_hat_lambda_orthogonalities}, also
\begin{equation*}
D_i=o_{\varepsilon_{1}}(A+\sum_{i} (A_i+L_i+D_i)).
\end{equation*}
Combining these estimates with \eqref{consm} establishes \eqref{uniqueness_condition} and therefore the desired uniqueness of an interior minimizer
as a critical point.
We finally turn to proving the \textit{almost}-orthogonalities (i) and (ii).
From \eqref{alpha_and_alpha_i_derivatives_estimate} and \eqref{lambda_hat_lambda_orthogonalities} we find
\begin{equation*}
\begin{split}
\int u^{\frac{4}{n-2}} & v \lambda_{i}\partial_{\lambda_{i}}\varphi_{a_{i},\lambda_{i}}d\mu_{g_{0}} \\
= \; &
O(\sum_{i}\frac{1}{\lambda_{i}^{n-2}}+\sum_{i\neq j}\varepsilon_{i,j}^{2} + \Vert v \Vert^{2})
+
O(\sum_{i}
\vert
\langle \varphi_{a_{i},\lambda_{i}},\lambda_{i}\partial_{\lambda_{i}}\varphi_{a_{i},\lambda_{i}}\rangle_{L_{g_{0}}} \vert^{2}
)
\end{split}	
\end{equation*}
and by simple expansions of $u^{\frac{4}{n-2}}$ for $u=\alpha u_{\infty} + \alpha^{i}\varphi_{a_{i},\lambda_{i}}+v$, that also
\begin{equation*}
\begin{split}
\int \varphi_{a_{i},\lambda_{i}}^{\frac{4}{n-2}} &  \lambda_{i}\partial_{\lambda_{i}}\varphi_{a_{i},\lambda_{i}} v d\mu_{g_{0}} \\
= \; &
O(\sum_{i}\frac{1}{\lambda_{i}^{n-2}}+\sum_{i\neq j}\varepsilon_{i,j}^{2} + \Vert v \Vert^{2})
+
O(\sum_{i}
\vert
\langle \varphi_{a_{i},\lambda_{i}},\lambda_{i}\partial_{\lambda_{i}}\varphi_{a_{i},\lambda_{i}}\rangle_{L_{g_{0}}} \vert^{2}
).
\end{split}	
\end{equation*}
Now assertion (i) of Lemma \ref{lem_L_g_0_of_bubble} follows from writing
\begin{equation*}
\begin{split}
\langle \lambda_{i}\partial_{\lambda_{i}}\varphi_{a_{i},\lambda_{i}},v\rangle_{L_{g_{0}}}
= \; &
4n(n-1)
\int \varphi_{a_{i},\lambda_{i}}^{\frac{4}{n-2}} \lambda_{i}\partial_{\lambda_{i}}\varphi_{a_{i},\lambda_{i}} v d\mu_{g_{0}} \\
& +
\int
(
\lambda_{i}\partial_{\lambda_{i}}L_{g_{0}}\varphi_{a_{i},\lambda_{i}}
-
4n(n-1)\lambda_{i}\partial_{\lambda_{i}}\varphi_{a_{i},\lambda_{i}}^{\frac{n+2}{n-2}}
)
v  d\mu_{g_{0}},
\end{split}
\end{equation*}
while assertion (ii) follows analogously, relying on \eqref{a_hat_a_orthogonalities} instead of
\eqref{lambda_hat_lambda_orthogonalities}.
\end{proof}

\extrafootertext{The authors have no conflict of interest to declare.
Data sharing is not applicable to this article as no datasets were generated or analysed during the current study.}


\begin{thebibliography}{99}

\bibitem{Ahmedou_Ben_Ayed_Non_Simple_Blow_Up}
{
Ahmedou M., Ben Ayed M.,
\emph{Non simple blow ups for the Nirenberg problem on half spheres.}
Disc. and Cont. Dyn. Systems, Vol. 42, No.12, p. 5967-6005
}

\bibitem{Aubin_Book}
{Aubin T.,
\emph{Some Nonlinear Problems in Riemannian Geometry.}
Springer Monographs in Mathematics,
Springer-Verlag Berlin Heidelberg 1998
}

\bibitem{Aubin_Bismuth}
{
Aubin T., Bismuth S.,
\emph{Courbure scalaire prescrite sur les variétés riemanniennes compactes dans le cas négatif.}
J. Funct. Anal. 143 (1997), No.2, 529-541
}

\bibitem{Amacha_Regbaoui}
{
Amacha I., Rachid Regbaoui R.,
\emph{Yamabe flow with prescribed scalar curvature.}
Pacific Journal of Mathematics,
Vol. 297 (2018), No.2, 257-275
}

\bibitem{Bahri_High_Dimensions}
{
Bahri A.,
\emph{An invariant for Yamabe type flows with applications
to scalar curvature problems in higher dimensions.}
Duke Math. J. 81 (1996), 323-466
}

\bibitem{Bahri_Coron_Three_Dimensional_Sphere}
{Bahri A., Coron J.M.,
\emph{The Scalar-Curvature problem on the
standard three-dimensional sphere.}
J. Funct. Anal.  95 (1991), 106-172
}

\bibitem{Bismuth_Dimension_Two}
{
Bismuth S.,
\emph{
Prescribed scalar curvature on a $C^{\infty}$
compact Riemannian manifold of dimension two.
}
Bull. Sci. math. 124 (2000), No.3, 239-248
}

\bibitem{Escobar_Schoen}
{
Escobar J., Schoen R.M.,
\emph{Conformal metrics
with prescribed scalar curvature.} Invent. Math., 86 (1986), 243-254
}

\bibitem{F}
{Friedman A.,
\emph{Partial differential equations of parabolic type.}
Robert E. Krieger Publishing Company, Malabar, Florida 1983
}

\bibitem{GHJL}
{Gover A.R., Hassannezhad A., Jakobson D., Levitin M.,
\emph{Zero and negative eigenvalues of the conformal Laplacian.}
J. Spectr. Theory 6 (2016), No.4, 793-806}

\bibitem{gunther}
{
G\"unther M.,
\emph{Conformal normal coordinates.}
Ann. Global Anal. Geom. 11 (1993), No.2, 173-184
}

\bibitem{Kazdan_Warner_JDE}
{
Kazdan J., Warner F.,
\emph{Scalar Curvature and conformal deformation of
Riemannian structure.}
J. Differ. Geom. 10 (1975) 113-134
}

\bibitem{lp}
{
Lee J., Parker T.,
\emph{The Yamabe problem.}
Bull. Amer. Math. Soc. (N.S.) 17 (1987), No.1, 37-91
}

\bibitem{MM1}
{
Malchiodi A., Mayer M.,
\emph{Prescribing Morse Scalar Curvatures: Blow-Up Analysis.}
Int. Mat. Res. Not., Vol. 2021 Issue 16, p 12532-12612
}

\bibitem{MM4}
{
Malchiodi A., Mayer M.,
\emph{Prescribing Morse scalar curvatures: pinching and Morse theory.}
Comm. in Pure and Appl. Math.,
vol 76, Issue 2, 406-450
}

\bibitem{Mayer_Scalar_Curvature_Flow}
{
Mayer M.,
\emph{A scalar curvature flow in low dimensions.}
Calc. Var. Partial Differential Equations 56 (2017), No.2, Art. 24, 41 pp
}

\bibitem{MM3}
{
Mayer M.,
\emph{Prescribing Morse scalar curvatures: Critical points at infinity.}
Advances in Calculus of Variations, vol. 15, No.2, 2022, pp. 151-190
}

\bibitem{MZ2}
{
Mayer, M., Zhu, C.,
\emph{Prescribing scalar curvatures: loss of minimizability in the negative Yamabe case.}
Preprint
}


\bibitem{Ouyang}
{
Ouyang T.,
\emph{On the positive solutions of semilinear equations $\Delta u+\lambda u^p=0$
on compact manifolds.}
Part II. Indiana Univ. Math. J. 40 (1991), 1083-1141
}

\bibitem{Pistoia_Roman}
{
Pistoia A., Roman C.,
\emph{Large conformal metrics with prescribed scalar curvature.}
J. Differential Equations 263 (2017) 5902-5938
}


\bibitem{Rauzy}
{
Rauzy A.,
\emph{Courbures scalaires des variétés d'invariant conforme négatif.}
Trans. Amer. Math. Soc. 347 (1995), No.12, 4729-4745
}

\bibitem{Rauzy_Multiplicity}
{
Rauzy A.,
\emph{Multiplicite pour un probleme de courbure scalaire prescrite.}
Bull. Sc. Math. 120 (1996) 153-194
}

\bibitem{Shubin_Spectral_Theory}
{
Shubin M.A.,
\emph{Spectral theory of elliptic operators on non compact manifolds.}
M\'ethodes semi-classiques Volume 1, \'Ecole d'Ét\'e, 207 (1992)  35-108.
}

\bibitem{Tang}
{
Tang T.,
\emph{Solvability of the equation $\Delta_gu+\tilde{S}u^\sigma=Su$ on manifolds.}
Proc. Am. Math. Soc. 121 (1994) 83-92
}

\bibitem{Vazquez_Veron}
{
Vazquez J.L., Vdron L.,
\emph{Solutions positives d'\'equations elliptiques semilin\'eaires sur des vari\'et\'es Riemanniennes compactes.}
C.R. Acad. Sci. Paris 312 (1991) 811-815
}




\end{thebibliography}
\end{document}